\documentclass[reqno,a4paper]{article}
\usepackage{amssymb,amsmath,epsfig,graphics,psfrag,graphicx,cite,color,subfigure}
\definecolor{lightblue}{rgb}{0.22,0.45,0.70}
\definecolor{darkgreen}{rgb}{0.22,0.75,0.40}
\definecolor{mygray}{rgb}{0.7,0.7,0.7}
\usepackage[colorlinks=true,breaklinks=true,linkcolor=lightblue,citecolor=lightblue]{hyperref}
\usepackage{float}

\numberwithin{equation}{section}
\numberwithin{figure}{section}
\numberwithin{table}{section}
\newcommand\cero{\boldsymbol{0}}

\newcommand\bL{\boldsymbol{L}}

\newcommand\bU{\mathbf{U}}
\newcommand\mV{\mathbb{V}}

\newcommand\ff{\boldsymbol{f}}
\newcommand\bg{\boldsymbol{g}}

\newcommand\bphi{\boldsymbol{\varphi}}

\newcommand\bu{\boldsymbol{u}}
\newcommand\bv{\boldsymbol{v}}
\newcommand\bx{\boldsymbol{x}}

\newcommand\bbbeta{\boldsymbol{\beta}}

\newcommand\cP{\mathcal{P}}
\newcommand\cT{\mathcal{T}}

\newcommand\cE{{\mathcal{E}}}

\newcommand\N{\mathbb{N}}
\newcommand\R{\mathbb{R}}

\newcommand{\norm}[1]{\left\|#1\right\|}

\newcommand{\set}[1]{\left\{#1\right\}}

\renewcommand\O{\Omega}
\newcommand\G{\Gamma}

\renewcommand\H{\mathrm{H}}
\renewcommand\L{\mathrm{L}}
\newcommand\Q{\mathrm{Q}}
\newcommand\Z{\mathbf{Z}}

\newcommand\LO{\L^2(\O)}
\newcommand\LOO{\L_{0}^2(\O)}

\newcommand\vdiv{\mathop{\mathrm{div}}\nolimits}

\newcommand\HsO{\H^s(\O)}

\newcommand\HusO{\H^{1+s}(\O)}

\newcommand\bp{\boldsymbol{p}}

\newcommand\ba{\boldsymbol{a}}

\newcommand\bn{\boldsymbol{n}}
\newcommand\bt{\boldsymbol{t}}

\newcommand\bomega{\boldsymbol{\omega}}
\newcommand\btheta{\boldsymbol{\theta}}

\newcommand\bxi{\boldsymbol{\xi}}
\newcommand\betta{\boldsymbol{\eta}}

\newtheorem{remark}{Remark}[section]
\newtheorem{lemma}{Lemma}[section]
\newtheorem{theorem}{Theorem}[section]

\newcommand{\dx}{\,\mbox{d}x}

\newcommand{\ds}{\,\mbox{d}s}
\newcommand\curl{\mathop{\mathbf{curl}}\nolimits}
\newcommand\rot{\mathop{\mathrm{rot}}\nolimits}

\newcommand{\supp}{\mathrm{supp}}

\def\CE{{\mathcal E}}
\def\CT{{\mathcal T}}
\def\CR{{\mathcal N}}

\newenvironment{proof}{\noindent{\it Proof.}}{\hfill$\square$}

\allowdisplaybreaks

\begin{document}

\title{Numerical analysis of a new formulation
for the Oseen equations \\ in terms of vorticity and Bernoulli pressure}

\author{Ver\'onica Anaya\thanks{GIMNAP, Departamento de Matem\'atica,
Universidad del B\'io-B\'io, Concepci\'on, Chile; and
Centro de Investigaci\'on en Ingenier\'ia Matem\'atica
(CI$^2$MA), Universidad de Concepci\'on, Concepci\'on, Chile. E-mail:
{\tt vanaya@ubiobio.cl}.},\quad 
David Mora\thanks{GIMNAP, Departamento de Matem\'atica, Universidad
del B\'io-B\'io, Concepci\'on, Chile; and
Centro de Investigaci\'on en Ingenier\'ia Matem\'atica
(CI$^2$MA), Universidad de Concepci\'on, Concepci\'on, Chile.
E-mail: {\tt dmora@ubiobio.cl}.},\quad 
Amiya K. Pani\thanks{Department of Mathematics,
Indian Institute of 
Technology Bombay, Powai, Mumbai 400076, India.
E-mail: {\tt akp@math.iitb.ac.in}},\quad 
Ricardo Ruiz-Baier\thanks{School of Mathematical Sciences,
Monash University, 9 Rainforest Walk, Clayton VIC 3800, Australia. E-mail:
{\tt ricardo.ruizbaier@monash.edu}.}}

\date{July 29, 2020}

\maketitle

\begin{abstract}
A variational formulation is introduced for the Oseen equations written in terms of vor\-ti\-city and Bernoulli pressure. The velocity is fully decoupled using the momentum balance equation, and it is later recovered by a post-process. A finite element method is also proposed, consisting in equal-order N\'ed\'elec finite elements and piecewise continuous polynomials for the vorticity and the Bernoulli pressure, respectively. The {\it a priori}  error analysis is carried out in the $\L^2$-norm for vorticity, pressure, and velocity; under a smallness assumption either on the convecting velocity, or on the mesh parameter.  Furthermore, an {\it a posteriori} error estimator is designed and its robustness and efficiency are studied using weighted norms. Finally, a set of numerical examples in 2D and 3D is given, where the error indicator serves to guide adaptive mesh refinement. These tests illustrate the behaviour of the new formulation in typical flow conditions, and they also confirm the theoretical findings.
\end{abstract}

\noindent
{\bf Key words}: Oseen equations; vorticity-based formulation; 
finite element methods; {\it a priori} error bounds;
{\it a posteriori} error estimation; numerical examples.

\smallskip\noindent
{\bf Mathematics subject classifications (2000)}:  65N30, 65N12, 76D07, 65N15

\maketitle
\section{Introduction}

In this paper, we propose a reformulation of the Oseen equations using only vorticity and Ber\-nou\-lli pressure. A similar splitting of the unknowns has been recently proposed in \cite{anaya17b} for the Brinkman equations. We extend those results for the Oseen problem and propose a residual-based {\it a posteriori} error estimator whose properties are studied using a weighted energy norm, as well as the $\L^2$-norm. 

There is an abundant body of literature dealing with numerical methods for incompressible flow problems using the vor\-ti\-city as a dedicated unknown. These include  spectral elements \cite{amoura07,azaiez06}, stabilised and least-squares schemes \cite{amara07,bochev99}, and mixed finite elements \cite{anaya15a,dubois03,duan03,salaun15,gatica11}, to name a few. 
Works specifically devoted to the analysis of numerical schemes for the Oseen equations in terms of vorticity include the non-conforming exponentially accurate least-squares spectral  method for Oseen equations  proposed in \cite{mohapatra16}, the least-squares method proposed in \cite{tsai05} for Oseen and Navier-Stokes equations, the family of vorticity-based first-order Oseen-type systems studied in \cite{chang07}, the enhanced accuracy formulation in terms of velocity-vorticity-helicity investigated in \cite{benzi12}, and the recent mixed and DG discretisations for Oseen's problem in velocity-vorticity-pressure form, proposed in  \cite{anaya18}.

The method advocated in this article focuses on N\'ed\'elec elements of order $k\ge 1$ for the vorticity and piecewise continuous polynomials of degree $k$, for the Bernoulli pressure. 
An abridged version of the analysis for this formulation has been recently advanced in \cite{altamirano}. In contrast, here we provide details on the 
{\it a priori} error estimates rigorously derived for the finite element discretisations in the $\L^2$-norm under enough regularity and under  smallness assumption on the mesh parameter.
Furthermore, we prove error estimates for two post-processes for the velocity
field in the $\L^2$-norm. The first one is similar to the one used in \cite{anaya17b} for Brinkman equations, which exploits 
the momentum equation and direct differentiation of the
discrete vorticity and Bernoulli pressure. For the second post-process
we solve an additional elliptic problem emanating from the constitutive equation defining vorticity, and it uses the continuity equation and 
the discrete vorticity appears on the right-hand side. This problem is discretised with, e.g., 
piecewise linear and continuous polynomials.

On the other hand, we address the construction of residual
based  {\it a posteriori} error estimators which are reliable and efficient. 
Adaptive mesh refinement strategies based on {\it a posteriori} error indicators
have a significant role in computing numerical solutions to partial differential equations, and this 
is of high importance in the particular context of incompressible flow problems. Robust and efficient error estimators permit to restore the optimal convergence of finite element methods, specifically when complex geometries or singular coefficients are present (which could otherwise lead to non-convergence or to the generation of spurious solutions) \cite{rv-1996}, and they can provide substantial enhancement to the accuracy of the approximations \cite{ain-ode}. {\it A posteriori} error analyses for vorticity-based equations are already available from the literature (see, e.g., \cite{agr2017,amara07,anaya15a,cgot2016}), 
but considering formulations substantially different to the one we put forward here. Our analysis complements these works by establishing upper and lower bounds in different 
norms, and using an estimator that is scaled according to the expected regularity of the solutions (which in turn also depends on the 
regularity of the domain).
Reliability of the {\it a posteriori} error estimator is proved
in the $\L^2$-norm, and local efficiency of the 
error indicator is shown by using a standard technique based on bubble functions.

We further remark that the present method has the advantage of direct computation of vorticity, and it is relatively competitive in terms of computational cost (for instance when compared with the classical MINI-element, or Taylor-Hood schemes). The type of vorticity-based formulations we use here can be of additional physical relevance in scenarios where boundary effects are critical, for example as in those discussed in \cite{davies01,olshanskii18}.  Moreover, the corresponding analysis is fairly simple, only requiring classical tools for elliptic problems. 

We have structured the contents of the paper in the following manner. We present the model problem as well as the two-field weak formulation and its solvability analysis in Section~\ref{sec:model}. The finite element discretisation is constructed in Section~\ref{sec:FE}, where we also derive the stability,
convergence bounds and we present two post-processes for the velocity field. Section~\ref{sec:apost} is devoted to the analysis of reliability and efficiency of a weighted residual-based {\it a posteriori} error indicator, and we close in Section~\ref{sec:numer} with a set of numerical tests that illustrate the properties of the proposed numerical scheme in a variety of scenarios, including validation of the adaptive refinement procedure guided by the error estimator.

\section{The continuous formulation of the Oseen problem}\label{sec:model}
This section deals with some preliminaries, variational formulation for the Oseen problem given in terms of vorticity and Bernoulli pressure  and its well-posedness.

\subsection{Preliminaries}
Let $\O\subset\R^3$ be a bounded and connected Lipschitz domain with 
its boundary $\G=\partial\O$ and 
further, let $\bn$ be 
the outward unit vector normal to $\G$. The starting point of our investigation  is the following form of the equations, that use velocity, vorticity, and Bernoulli pressure (see, e.g., \cite{anaya18,ORsisc02})
\begin{align}\label{NS-cont-vort}\begin{split}
\sigma\bu-\nu\Delta\bu+\curl\bu\times\bbbeta +\nabla p & =
  \ff \qquad \mbox{ in } \O, \\
  \vdiv\bu & =  0 \qquad\, \mbox{ in } \O,
 \end{split}\end{align}
where $\nu>0$ is the kinematic viscosity, and a linearisation and backward Euler time stepping 
explain the terms $\sigma>0$ as the inverse of the time step, and $\bbbeta$ as 
 an adequate approximation of velocity (representing for example the velocity 
 at a previous time step). The Bernoulli pressure relates to the true fluid pressure $P$ 
 as follows $p := P + \frac{1}{2}\bu\cdot\bu-\lambda$, where $\lambda$ is the mean value of $\frac{1}{2}\bu\cdot\bu$.

The structure of \eqref{NS-cont-vort} suggests to introduce the rescaled
vorticity vector $\bomega:=\sqrt{\nu}\curl\bu$ as a new
unknown. Thus, the Oseen problem can be formulated as: Find $\bu,\bomega,p$ such that
\begin{align}\label{eq:momentum}
 \sigma\bu +\sqrt{\nu}\curl \bomega + \nu^{-1/2} \bomega \times {\boldsymbol\beta}
 +\nabla p\,=&\,  \ff & \mbox{ in } \O, \\
  \bomega-\sqrt{\nu}\curl\bu= &\,\cero & \mbox{ in } \O,\label{eq:constitutive} \\
  \vdiv\bu  = &\, 0 & \mbox{ in } \O,  \label{eq:mass}\\
  \bu  = &\,\bg &    \mbox{ on } \G.\label{eq:bc}
 \end{align}
The vector of external forces $\ff$ absorbs the contributions
related to previous time steps and to the fixed states in
the linearisation procedure that leads  from Navier-Stokes to Oseen equations. 
Along with to the Dirichlet boundary condition
for the velocity on $\G$, 
the additional condition $(p,1)_{\O,0} = 0$ is required to have uniqueness of the Bernoulli pressure. 
We will also assume that the data are
regular enough: $\ff\in\L^2(\O)^3$ and $\bbbeta\in \L^{\infty}(\O)^3.$
However, we do not restrict the behaviour of $\vdiv{\boldsymbol\beta}$. 
For different assumptions on $\bbbeta$ we refer to, e.g., 
\cite{barrios,Cockburn_JSC13,Cockburn_Mcomp04,tsai05}.

For the sake of conciseness of the presentation,
the analysis in the sequel is carried out for homogeneous boundary conditions on velocity, 
i.e. $\bg=\cero$ on $\Gamma$. Non homogeneous boundary data, as well as mixed boundary conditions, 
will be considered in the numerical examples in Section~\ref{sec:numer}, below. 

\subsection{Variational formulation}
For any $s\geq 0$, the symbol
$\norm{\cdot}_{s,\O}$ denotes
the norm of the Hilbert  Sobolev spaces $\HsO$ or
$\HsO^3$, with the usual convention $\H^0(\O):=\LO$. For $s\geq 0$, we recall
the definition of the space
$$\H^s(\curl;\O):=\set{\btheta\in\HsO^3:\ \curl\btheta\in\HsO^3},$$ endowed with the norm
 $\norm{\btheta}_{\H^s(\curl;\O)}=\Big(\norm{\btheta}_{s,\O}^2+\norm{\curl\btheta}^2_{s,\O}\Big)^{1/2}$, and will denote
$\H(\curl;\O)=\H^0(\curl;\O)$. Finally, $c$ and
$C$, with subscripts, tildes, or hats, will represent a generic constant
independent of the mesh parameter $h$.

We denote the function spaces  
\begin{gather*}
\Z:=\H(\curl;\O), \qquad \text{and} \qquad
\Q:=\H^{1}(\O)\cap\LOO,
\end{gather*}
which are endowed, respectively, with the following norms  
$$\Vert\btheta\Vert_{\Z}:=\left(\Vert\btheta\Vert_{0,\O}^2+
\nu\Vert\curl\btheta\Vert_{0,\O}^2\right)^{1/2} \qquad \mbox{and}\qquad 
\Vert q\Vert_{\Q}:= ( \Vert q\Vert^2_{0,\O}+
\Vert \nabla q\Vert^2_{0,\O} )^{1/2}
.$$
Here, $\LOO$ represents the set of $\L^2(\O)$ functions with mean value zero.

In addition, for sake of the subsequent analysis,
it is convenient to introduce the following space 
$$ \mV := \{ (\btheta,q)\in \L^2(\Omega)^3\times \L^2_0(\Omega): \ \sqrt{\nu} \curl \btheta + \nabla q \in \L^2(\Omega)^3\}.$$
\begin{lemma}
The space $\mV$ endowed with the norm defined by 
\begin{equation}\label{eq:normV}
\|(\btheta,q)\|_{\mV} : = \Big(\sigma\Vert\btheta\Vert_{0,\O}^2+\Vert\sqrt{\nu}\curl\btheta+\nabla q\Vert_{0,\O}^2
+\Vert q\Vert_{0,\O}^2\Big)^{1/2}
\end{equation}
is a Hilbert space. 
\end{lemma}
\begin{proof} Note that \eqref{eq:normV} is in fact a norm  as $\|(\btheta,q)\|_{\mV}=0$ implies $(\btheta,q)=(\cero,0)$ a.e.
Now, it is easy to check that the norm satisfies the parallelogram identity and hence, it induces an inner product  by
the polarisation identity. Therefore,  $\mV$  equipped with this inner product is an inner product space.   To complete the proof, it remains to show that this space is complete. To this end, let $\{(\btheta_n,q_n)\}_{n\in \mathbb{N}}$  be an arbitrary Cauchy sequence in $\mV.$  From  the completeness of $\L^2(\Omega)^3$ and $\L^2_0(\Omega)$, it follows that 
$(\btheta_n,q_n)\to (\hat{\btheta},\hat{q}) \in \L^2(\Omega)^3\times \L^2_0(\Omega)$ and $\sqrt{\nu} \curl {\btheta}_n + \nabla {q_n} \to \bphi \in \L^2(\Omega)^3.$
We now observe that 
 $\sqrt{\nu} \curl \hat{\btheta} + \nabla \hat{q} \in \mathcal{D}(\Omega)'$, and  for $\bxi\in  \mathcal{D}(\Omega)$
\begin{align*}
\langle \sqrt{\nu} \curl \hat\btheta + \nabla \hat{q} , \bxi\rangle & = \sqrt{\nu} \langle \curl \hat{\btheta},\bxi\rangle + \langle \nabla \hat{q}, \bxi\rangle =  \sqrt{\nu}\langle \hat{\btheta} ,\curl \bxi\rangle - \langle \hat{q},\vdiv\bxi\rangle\\
 & = \lim_{n\to\infty}\bigl[ \sqrt{\nu}\langle \hat{\btheta}_n ,\curl \bxi\rangle - \langle \hat{q}_n,\vdiv\bxi\rangle \bigr]\\
 & =  \lim_{n\to\infty}\bigl[\sqrt{\nu} \langle \curl \hat{\btheta}_n,\bxi\rangle + \langle \nabla \hat{q}_n, \bxi\rangle \bigr]\\
 & = \langle \bphi ,\bxi \rangle.
 \end{align*}
Here we have employed integration by parts twice and the continuity of the involved operators in 
passing to the limit. Therefore, $\sqrt{\nu} \curl \hat\btheta + \nabla \hat{q}= \bphi \in \L^2(\Omega)^3$ and this 
completes the rest of the proof.
\end{proof}

In order to derive a variational formulation of  the problem, we 
 test \eqref{eq:constitutive} against a sufficiently smooth function $\sigma\btheta$.
Then, integrating by parts (using the classical curl-based Gauss theorems from, e.g.,  \cite{gr-1986}) and using the velocity boundary condition, we arrive at
\begin{equation}\label{step1}
\sigma \int_{\O}\bomega\cdot\btheta-\sigma \sqrt{\nu}\int_{\O}\bu\cdot\curl\btheta
=0.
\end{equation}
Next, from the momentum equation \eqref{eq:momentum}, we readily obtain the relation 
\begin{equation}\label{repu}
\sigma\bu=\ff-\sqrt{\nu}\curl\bomega -\nu^{-1/2}\bomega \times {\boldsymbol\beta}
-\nabla p\quad\text{in}\quad\O,
\end{equation}
and after replacing \eqref{repu} in \eqref{step1},
we find that
\begin{align*}
\sigma\int_{\O}\bomega\cdot\btheta+\nu\int_{\O}\curl\bomega\cdot\curl\btheta
&+\sqrt{\nu}\int_{\O}\nabla p\cdot\curl\btheta
+\int_{\O} (\bomega \times {\boldsymbol\beta}) \cdot\curl\btheta=\sqrt{\nu}\int_{\O}\ff\cdot\curl\btheta.
\end{align*}
Next for a given sufficiently smooth function
 $q$, we can test \eqref{eq:momentum} against $\nabla q$. 
Then, we integrate by parts and use again the velocity boundary condition, as well as \eqref{eq:mass}
to arrive at 
$$\displaystyle \sigma\int_{\O}\bu\cdot\nabla q= 0,$$
which leads to the variational form 
\begin{equation*}
\sqrt{\nu}\int_{\O}\curl\bomega\cdot\nabla q +\nu^{-1/2}\int_{\O}
(\bomega \times {\boldsymbol\beta}) \cdot \nabla q + \int_{\O} \nabla p \cdot \nabla q =
\int_{\O}\ff\cdot\nabla q.
\end{equation*}
Summarising, problem \eqref{eq:momentum}-\eqref{eq:bc} is
written in its weak form as: Find $(\bomega,p)\in \mV$
such that
\begin{equation}\label{WVF3}
\mathcal{A}((\bomega,p),(\btheta,q))= \mathcal{F}(\btheta,q)\quad\forall(\btheta,q)\in\mV,
\end{equation}
where the multilinear form $\mathcal{A}:\mV\times\mV\to\R$ and linear
functional $\mathcal{F}:\mV\to\R$
are specified as 
\begin{align}
\mathcal{A}((\bomega,p),(\btheta,q)) & := \sigma\int_{\O}\bomega\cdot\btheta
+\int_{\O}(\sqrt{\nu}\curl\bomega+\nabla p)\cdot(\sqrt{\nu}\curl\btheta+\nabla q)\label{bili1}\\
&\qquad   +\nu^{-1/2}\int_{\O} (\bomega \times {\boldsymbol\beta})\cdot(\sqrt{\nu}\curl\btheta+\nabla q),\nonumber\\
\mathcal{F}(\btheta,q) &:=\int_{\O}\ff\cdot(\sqrt{\nu}\curl\btheta+\nabla q)
\label{functi1}.
\end{align}

While our whole development will focus on this vorticity-pressure formulation, we stress that 
from \eqref{repu} we can immediately have an expression for  velocity 
\begin{equation}\label{repu2}
\bu=\sigma^{-1}\left(\ff-\nu^{-1/2}\bomega \times {\boldsymbol\beta}-(\sqrt{\nu}\curl\bomega+\nabla p)\right)\quad\text{in}\quad\O.
\end{equation}

\begin{remark}
The reason for scaling the vorticity with $\sqrt{\nu}$ is now apparent from the structure of the 
variational form in \eqref{bili1}. On the other hand, if we write instead $\bar{\bomega}:=\curl\bu$, then \eqref{WVF3} 
could be written as: Find $(\bar{\bomega},p)\in \mV$ such that
\begin{align*}
\frac{\sigma}{\nu}\int_{\O}\bar{\bomega}\cdot\btheta &+\int_{\O}(\curl\bar{\bomega}+\nabla p)\cdot(\curl\btheta+\nabla q)
+\nu^{-1}\int_{\O} (\bar{\bomega} \times {\boldsymbol\beta})\cdot(\curl\btheta+\nabla q)\\
& =\int_{\O}\ff\cdot(\curl\btheta+\nabla q)\;\;\;\;\forall(\btheta,q)\in\mV,
\end{align*}
and the analysis of this problem 
follows the same structure as that of \eqref{WVF3}. 
\end{remark}


\noindent Let us first provide an auxiliary result to be used in the derivation of
{\it a priori} error estimates.
\begin{lemma}\label{acot-norma}
The multilinear form $\mathcal{A}$ satisfies the following bounds for all $(\btheta,q)\in\mV$,
\begin{align}\label{norma}
 \mathcal{A}((\btheta,q),(\btheta,q)) & \geq \sigma\left(1-\frac{2 \Vert \bbbeta \Vert_{\infty,\O}^2}{\nu\sigma}\right)
\Vert \btheta\Vert_{0,\O}^{2}
+\frac{1}{2}  \Vert\sqrt{\nu}\curl\btheta
+\nabla q\Vert_{0,\O}^2,\\
 \mathcal{A}((\bomega,p),(\btheta,q)) & \leq \Vert(\bomega,p)\Vert_{\mV} \Vert(\btheta,q)\Vert_{\mV} .\label{norma2}
\end{align}
\end{lemma}
\begin{proof}
From the definition of $\mathcal{A}(\cdot,\cdot),$ we readily obtain the relation
$$\sigma\Vert \btheta\Vert_{0,\O}^{2}+\Vert\sqrt{\nu}\curl\btheta+\nabla q\Vert_{0,\O}^{2}
+\nu^{-1/2}\int_{\O} (\btheta \times {\boldsymbol\beta})\cdot(\sqrt{\nu}\curl\btheta+\nabla q)
=\mathcal{A}((\btheta,q),(\btheta,q)).$$
Subsequently, an appeal to the Cauchy-Schwarz inequality leads to 
\begin{equation}\label{eqrf}
\sigma\left(1-\frac{2 \Vert \bbbeta \Vert_{\infty,\O}^2}{\nu\sigma}\right)\Vert \btheta\Vert_{0,\O}^{2}
+\frac{1}{2}\Vert\sqrt{\nu}\curl\btheta+\nabla q\Vert_{0,\O}^{2}
\le \mathcal{A}((\btheta,q),(\btheta,q)).
\end{equation}
Thus, \eqref{norma} follows from \eqref{eqrf}, and 
relation \eqref{norma2} follows directly form the Cauchy-Schwarz inequality.  This completes the proof.
\end{proof}

As a consequence of Lemma \ref{acot-norma}, we can readily derive the following  result, stating the stability of problem  \eqref{WVF3}. 
\begin{lemma}\label{lem:stabi}
Assume that 
\begin{equation}\label{bound1}
2 \Vert {\boldsymbol\beta} \Vert_{\infty,\O}^2< \nu\sigma,
\end{equation}
holds true. 
Then, there exists $C >0$ such that 
$$\Vert( \bomega,p)\Vert_{\mV} \le C \Vert\ff\Vert_{0,\O}.$$
\end{lemma}

\begin{proof} Choose $(\btheta,q)=(\bomega,p)$ in (\ref{WVF3}). From (\ref{norma}) with (\ref{bound1}), and the bound 
$$|\mathcal{F}(\bomega,p) | \leq \Vert\ff\Vert_{0,\O} \;\Vert\sqrt{\nu}\curl\bomega+\nabla p\Vert_{0,\O} \leq \Vert\ff\Vert_{0,\O} \;
\|(\bomega,p)\|_{\mV},$$ 
we obtain
$$\Big(\Vert \bomega\Vert_{0,\O}^2+\Vert\sqrt{\nu}\curl\bomega+\nabla p\Vert_{0,\O}^2\Big) \leq C\; \Vert\ff\Vert_{0,\O} \;
\|(\bomega,p)\|_{\mV}.$$
Then, we note that 
$$\Vert p \Vert_{0,\O} \leq C_{s} \; \Vert \nabla p\Vert_{{-1},\O},
$$
and invoking the definition of the $\H^{-1}$-norm, it is  observed that
\begin{align*}
\Vert \nabla p\Vert_{{-1},\O} 
&\leq \sup_{\{q\in \H^1_0(\O): \Vert \nabla q\Vert_{0,\O} =1\}} (\nabla p , \nabla q)_{0,\O} \\
&= \sup_{\{q\in \H^1_0(\O): \Vert \nabla q\Vert_{0,\O} =1\} } \Big( (\sqrt{\nu}\curl\bomega+\nabla p  , \nabla q)_{0,\O} - (\sqrt{\nu}\curl\bomega, \nabla q)_{0,\O}\Big).
\end{align*}
Using integration by parts for the second term in this last relation, and using that 
$q\in \H^1_0(\Omega)$ as well as $\nabla\cdot \curl \bomega=0$,
$(\sqrt{\nu}\curl\bomega, \nabla q)=0$, we end up with 
\begin{align}\label{estL2:q}
\Vert p \Vert_{0,\O} \leq C_s \Vert \nabla p\Vert_{{-1},\O} 
&\leq  C_s  \;  \Vert \sqrt{\nu}\curl\bomega+\nabla p\Vert_{0,\O}\nonumber \\
&\leq  C_s \;\|(\bomega,p)\|_{\mV}.
\end{align}
Altogether, it completes the rest of the proof.
\end{proof}

\begin{theorem}
Under the assumption \eqref{bound1},  there exists a unique weak solution $(\bomega,p)\in \mV$ to the problem \eqref{WVF3}, which depends continuously on $\ff$.
\end{theorem}

\begin{proof}  The continuous dependence of the solution  $(\bomega,p) \in  \mV$ on the given data $\ff$ is a consequence of 
the stability Lemma~\ref{lem:stabi}. Likewise, a straightforward application of that result implies the uniqueness of solution. 

On the other hand, for the existence we note that  the multilinear form  $ \mathcal{A}(\cdot,\cdot) $ is both coercive and bounded  in $\mV$ with respect to 
$\|(\cdot,\cdot)\|_{\mV}$  because of Lemmas \ref{acot-norma}
and \ref{lem:stabi}. Therefore, an appeal to the Lax-Milgram Lemma completes the rest of the proof.
\end{proof}

\begin{remark}
Even if $\bbbeta$ violates \eqref{bound1} we can still address the well-posedness of problem  \eqref{WVF3}.
Since problem \eqref{NS-cont-vort} with the boundary conditions $\bu = \cero$ on $\Gamma$ is equivalent to \eqref{WVF3} under the assumption of 
sufficient regularity, then the unique solvability of \eqref{NS-cont-vort} implies that of \eqref{WVF3}. Now, denoting 
by $\mathbf{P}$ the Leray projection operator that maps $\bL^2$ onto a divergence-free space, 
we can see that the following problem 
$${\mathcal{L}} (\bu):= {\mathbf{P}}( -\nu \Delta \bu + \curl \bu \times \bbbeta + \sigma\bu + \nabla p) = {\mathbf{P}} \ff, $$
defines a Fredholm alternative (see for instance, \cite{carstensen16}). Note also that, as long as  zero is not in the spectrum of ${\mathcal{L}}$,  the operator  ${\mathcal{L}}$ is invertible. With the null space of ${\mathcal{L}}$ being a trivial space,
the operator  ${\mathcal{L}}$   is indeed an isomorphism onto the dual space $\Z'$ of $\Z.$ Finally,  $p$ is recovered in a standard way. 
 \end{remark}
In any case, for the rest of the paper we will simply assume that 
 \begin{itemize}
 \item[\textbf{(A)}] The problem  \eqref{WVF3} has a unique weak solution $(\bomega,p) \in \mV.$
\end{itemize}

\section{Finite element discretisation and error estimates} \label{sec:FE}
This section focuses on finite element approximations and their {\it a priori} error estimates.

\subsection{Galerkin scheme and solvability}
Let $\{\cT_{h}(\O)\}_{h>0}$ be a shape-regular
family of partitions of the polyhedral region
$\bar\O$, by tetrahedrons $T$ of diameter $h_T$, with mesh size
$h:=\max\{h_T:\; T\in\cT_{h}(\O)\}$.
In what follows, given an integer $k\ge1$ and a subset
$S$ of $\R^3$, $\cP_k(S)$ will denote the space of polynomial functions
defined locally in $S$ and being of total degree $\leq k$.

Now, for any $T\in\cT_{h}(\O)$ we recall the definition of the local N\'ed\'elec space
$$\N_k(T):=\cP_{k-1}(T)^3\oplus R_{k}(T),$$
where $R_{k}(T):=\{\bp\in\bar{\cP}_{k}(T)^3: \bp(\bx)\cdot\bx=0\}$, and where
$\bar{\cP}_{k}$ is the subset of homogeneous polynomials of degree $k$.
With this we define the discrete spaces for vorticity and Bernoulli pressure:
\begin{align}
\Z_h&:=\{\btheta_h\in\Z: \btheta_h|_T\in\N_k(T)\quad\forall T\in\cT_{h}(\O)\},\nonumber\\
\Q_h&:=\{q_h\in\Q: q_h|_T\in\cP_k(T)\quad\forall T\in\cT_{h}(\O)\},\label{space2}\\
\mV_h&:=\Z_h \times \Q_h,\nonumber
\end{align}
and remark that functions in $\Z_h$ have continuous tangential components across
the faces of $\cT_{h}(\O)$.

Let us recall that for $s>1/2$, the N\'ed\'elec global interpolation operator
$\CR_{h}:\H^s(\curl;\O)\to\Z_h$ (cf. \cite{Alo-Valli}), satisfies 
the following approximation property:
%
For all $\btheta\in\H^{s}(\curl;\O)$ with $s\in(1/2,k]$,
there exists $C_{\mathrm{apx}} >0$ independent of $h$, such that
\begin{equation}\label{prop2N}
\Vert \btheta-\CR_{h}\btheta\Vert_{\Z}\le
C_{\mathrm{apx}}\;h^{s}\Vert\btheta\Vert_{\H^{s}(\curl;\O)}.
\end{equation}
On the other hand, for all $s>1/2$, the usual Lagrange interpolant
$\Pi_h:\HusO\cap\Q\to\Q_h$ features a similar property. Namely: 
For all $q\in\H^{1+s}(\O)$, $s\in(1/2,k]$
there exists $C_{\mathrm{apx}} >0$, independent of $h$, such that
\begin{equation}\label{prop2pi}
\Vert q-\Pi_{h}q\Vert_{\Q}\le
C_{\mathrm{apx}}\;h^{s}\Vert q\Vert_{\H^{1+s}(\O)}.
\end{equation}

The  Galerkin approximation of \eqref{WVF3} reads: 
{\em Find $(\bomega_h,p_h)\in\mV_h$ such that}
\begin{equation}\label{WVF3dis}
\mathcal{A}((\bomega_h,p_h),(\btheta_h,q_h))= \mathcal{F}(\btheta_h,q_h)\quad\forall(\btheta_h,q_h)\in\mV_h,
\end{equation}
where the multilinear form $\mathcal{A}:\mV_h \times\mV_h\to\R$ and the linear
functional $\mathcal{F}:\mV_h\to\R$ are specified as in \eqref{bili1} and \eqref{functi1}, respectively.

Next, let us prove that the discrete formulation \eqref{WVF3dis} is well-posed. 
 
 Before that, we address the stability of the 
 discrete problem. 
\begin{lemma}\label{stability-lemma}
Under the assumption $(A)$, and $h>0$ small enough,
 there exists $C>0,$ independent of $h,$ such that 
\begin{equation*}
\|(\bomega_h,p_h)\|_{\mV} \leq C\;\Vert \ff\Vert_{0,\O}.
\end{equation*}
\end{lemma}
\begin{proof}
Choosing $(\btheta_h,q_h) = (\bomega_h,p_h)$ in  \eqref{WVF3dis}, a use of the Cauchy-Schwarz inequality with 
the estimate
\begin{align*}
\nu^{-1/2}\int_{\O} (\bomega_h \times {\boldsymbol\beta})\cdot(\sqrt{\nu}\curl\bomega_h+\nabla p_h) 
\leq & 2 \; \nu^{-1/2}  \; \Vert \boldsymbol\beta \Vert_{\infty,\O} \Vert  \bomega_h \Vert_{0,\O}\Vert \sqrt{\nu}\curl\bomega_h+\nabla p_h \Vert_{0,\O},
\end{align*}
yields
\begin{align}\label{stab02}
\sigma\Vert\bomega_h\Vert_{0,\O}^2+\Vert\sqrt{\nu}\curl\bomega_h+\nabla p_h\Vert_{0,\O}^2 \leq 
\Big(\Vert \ff\Vert_{0,\O} + 2 \; \nu^{-1/2}   \Vert \boldsymbol\beta \Vert_{\infty,\O} \Vert  \bomega_h \Vert_{0,\O}\Big)\; \Vert \sqrt{\nu}\curl\bomega_h+\nabla p_h \Vert_{0,\O}.
\end{align}
By \eqref{estL2:q}, it follows that
\begin{align*}
 \Vert p_h\Vert_{0,\O} \leq  C_s \;\Vert\sqrt{\nu}\curl\bomega_h
 +\nabla p_h\Vert_{0,\O}.
\end{align*}
And eventually we arrive at 
\begin{equation}\label{stab03}
\|(\bomega_h,p_h)\|_{\mV}
\leq C\; \Big(\Vert \ff\Vert_{0,\O} +\sigma^{-1/2}\; \nu^{-1/2}  \sigma^{1/2}  \Vert {\boldsymbol\beta} \Vert_{\infty,\O} \Vert  \bomega_h \Vert_{0,\O}\Big).
\end{equation}

In order to complete the proof, we require an estimate  for $\sigma^{1/2}\Vert  \bomega_h \Vert_{0,\O}$. For this we apply the Aubin-Nitsche duality argument 
to the following adjoint problem: 
Find $(\widetilde{\bomega},\widetilde{p})\in\Z\times\Q$ such that
\begin{equation*}
\mathcal{A}((\btheta,q),(\widetilde{\bomega},\widetilde{p}))= \sigma (\bomega_h, \btheta )_{0,\O}+ (p_h, q)_{0,\O}  \quad \quad \forall(\btheta, q)\in \Z\times\Q,
\end{equation*}
and whose solution (after assuming the natural additional regularity $\widetilde{\bomega}\in\H^{\delta}(\curl;\O)$ and $\widetilde{p}\in\H^{1+\delta}(\O)$,
for some $\delta\in(1/2,1]$) 
satisfies 
\begin{equation*}
\|\widetilde{\bomega}\|_{\H^{\delta}{(\curl,\O)}} + \|\widetilde{p}\|_{\H^{1+\delta}{(\O)}}
\leq C_{\mathrm{reg}} ( \sigma^{1/2} \|\bomega_h\|_{0,\O} + \|p_h\|_{0,\O}), 
\end{equation*}  
for $C_{\mathrm{reg}}>0$ a uniform regularity constant. 
Then, we set $(\btheta,q) = (\bomega_h,p_h)$ and find out
that for all $(\btheta_h,q_h) \in \Z_h\times\Q_h$, the following relation holds:  
\begin{align*}
&\sigma\;\Vert \bomega_h\Vert_{0,\O}^2 + \Vert  p_h\Vert_{0,\O}^2  = 
\mathcal{A}((\bomega_h,p_h),(\widetilde{\bomega},\widetilde{p}) ) 
  =  \mathcal{A}((\bomega_h,p_h),(\widetilde{\bomega}-\btheta_h,\widetilde{p}-q_h) ) - \mathcal{F}(\widetilde{\bomega}-\btheta_h,\widetilde{p}-q_h)+\mathcal{F}(\widetilde{\bomega},\widetilde{p}) \nonumber \\
 &\qquad  \leq \sigma \Vert \bomega_h\Vert_{0,\O}\Vert \widetilde{\bomega}-\btheta_h\Vert_{0,\O} +\Big(  \Vert \sqrt{\nu}\curl \widetilde{\bomega}+\nabla\widetilde{ p} \Vert_{0,\O}  
 +\Vert \sqrt{\nu}\curl(\widetilde{\bomega}-\btheta_h)+\nabla (\widetilde{p} - q_h) \Vert_{0,\O} \Big)\;\Vert\ff\Vert_{0,\O} \nonumber \\
 &\qquad \quad + \Big(
 \Vert \sqrt{\nu}\curl \bomega_h+\nabla p_h \Vert_{0,\O} +  \nu^{-1/2}   \Vert {\boldsymbol\beta }\Vert_{\infty,\O} \Vert  \bomega_h \Vert_{0,\O}\Big)\;
\Vert \sqrt{\nu}\curl(\widetilde{\bomega}-\btheta_h)+\nabla (\widetilde{p} - q_h) \Vert_{0,\O} \nonumber \\
& \qquad \leq C_{\mathrm{reg}}\;C_{\mathrm{apx}}\; 
\biggl(\sigma \Vert \bomega_h\Vert_{0,\O}^2 + \Vert  p_h\Vert_{0,\O}^2 \biggr)^{1/2} 
\Big( h^{\delta}\; \big(1+\sigma^{1/2} +  \sigma^{-1/2}\;\nu^{-1/2}   \Vert {\boldsymbol\beta} \Vert_{\infty,\O}\big)  \;\Vert(\bomega_h,p_h)\Vert_{\mV} + \Vert\ff\Vert_{0,\O} \Big).
\end{align*}
In this way, 
we obtain 
\begin{equation}\label{stab06}
\biggl(\sigma \Vert \bomega_h\Vert_{0,\O}^2 + \Vert  p_h\Vert_{0,\O}^2 \biggr)^{1/2} \leq    C_{\mathrm{reg}}\;C_{\mathrm{apx}}\; \Big(h^{\delta}\; \Big(1+\sigma^{1/2} + \sigma^{-1/2}\;\nu^{-1/2}   \Vert {\boldsymbol\beta }\Vert_{\infty,\O}\Big)\; \Vert (\bomega_h,p_h)\Vert_{\mV} + \Vert\ff\Vert_{0,\O} \Big),
\end{equation}
and on substitution of \eqref{stab06} into \eqref{stab03}, we readily see that
$$\biggl(1-C_s\; C_{\mathrm{reg}}\;C_{\mathrm{apx}}\; \sigma^{-1/2}\;\nu^{-1/2} \;\Big(1+\sigma^{1/2} + \sigma^{-1/2}\;\nu^{-1/2}   \Vert \boldsymbol\beta \Vert_{\infty,\O}\Big)\;h^{\delta} \biggr)\;
 \Vert (\bomega_h,p_h)\Vert_{\mV} \leq C\; \Vert \ff\Vert_{0,\O}.$$
Therefore, there is a positive $h_0$ such that for  $0<h\leq h_0$,  the following holds:
$$  \biggl(1-C_s\; C_{\mathrm{reg}}\;C_{\mathrm{apx}}\; \sigma^{-1/2}\;\nu^{-1/2} \;\Big(1+\sigma^{1/2} + \sigma^{-1/2}\;\nu^{-1/2}   \Vert \boldsymbol\beta \Vert_{\infty,\O}\Big)\;h^{\delta} \biggr)\; \geq \gamma_0>0,$$
for some positive $\gamma_0,$ independent of $h$. This completes the rest of the proof. 
\end{proof}

\begin{theorem}
For  $h>0$ small enough, 
the discrete problem \eqref{WVF3dis} has a unique solution $(\bomega_h,p_h)\in\mV_h$.
\end{theorem}
\begin{proof}
Since the assembled discrete problem \eqref{WVF3dis} is a square linear system,
it is enough to establish uniqueness of solution. Considering $\ff=\cero$ and using $(\btheta_h,q_h):=(\bomega_h,p_h)$ as a test 
function in \eqref{WVF3dis},  the discrete stability result in Lemma \ref{stability-lemma} (which is valid assuming \eqref{bound1})
immediately implies that $\bomega_h=\cero$ and $p_h=0$, thus concluding the proof.
\end{proof}

\begin{remark}
When the  condition  \eqref{bound1} is satisfied,  we modify the stability proof of Lemma~\ref{stability-lemma} as follows:
From \eqref{stab02},  using the Young inequality on the right-hand side
\begin{align*}
\sigma\Big(1-\frac{2\Vert \boldsymbol\beta \Vert_{\infty,\O}^2}{\sigma \nu}\Big)\Vert\bomega_h\Vert_{0,\O}^2+\frac{1}{2}
\Vert\sqrt{\nu}\curl\bomega_h+\nabla p_h\Vert_{0,\O}^2 \leq 
\Vert \ff\Vert_{0,\O}^2. 
\end{align*}
In addition, applying Young's inequality once again, and appealing to \eqref{estL2:q}, we obtain the desired  stability result.

Note that in this case, we do not need a smallness condition on the mesh parameter $h$.
\end{remark}


\subsection{{\it A priori} error estimates}
In this subsection, using a classical duality argument we 
bound the error measured in the 
$\L^2$-norm by the error in the  norm $\|(\cdot,\cdot)\|_{\mV}$. Then, we 
establish an energy error estimate that eventually yields 
an optimal bound in $\L^2$. 

Let $(\bomega,p)\in\mV$ and $(\bomega_h,p_h)\in\mV_h$ be the unique
solutions to the continuous  and discrete problems (cf. \eqref{WVF3} and \eqref{WVF3dis}), respectively. 
Then, we obtain  
\begin{equation}\label{errorequation}
\mathcal{A}((\bomega-\bomega_h,p-p_h),(\btheta_h,q_h))=0\quad\forall(\btheta_h,q_h)\in\mV_h.
\end{equation}

\begin{lemma}[An $\L^2$-estimate]\label{converl2}
There exists $C>0$, independent of $h$, such that for 
$h$ small enough, and 
$\delta\in (1/2,1]$
$$
\Vert\bomega-\bomega_h\Vert_{0,\O}+\Vert p-p_h\Vert_{0,\O}\le
C\;h^{\delta}\; \Vert (\bomega-\bomega_h, p-p_h) \Vert_{\mV}.
$$ 
\end{lemma}
\begin{proof}
We appeal again to the Aubin-Nitsche duality argument. 
For this, let us  
consider the adjoint continuous problem:
Find $(\widetilde{\bomega},\widetilde{p})\in\mV$ such that
\begin{equation}\label{estima2}
\mathcal{A}((\btheta,q),(\widetilde{\bomega},\widetilde{p}))= ( \sigma (\bomega-\bomega_h), \btheta )_{0,\O} + (p-p_h, q)_{0,\O} \quad \quad \forall(\btheta, q)\in \mV.
\end{equation}
In addition, let us suppose that \eqref{estima2} is well-posed
and that $\widetilde{\bomega}\in \H^{\delta}(\curl;\O)$
and $\widetilde{p}\in \H^{1+\delta}(\O)$, and there exists a constant $C_{\mathrm{reg}}>0$,
such that 
\begin{equation}\label{elliptic-regularity}
\|\widetilde{\bomega}\|_{\H^{\delta}{(\curl,\O)}} + \|\widetilde{p}\|_{\H^{1+\delta}{(\O)}}
\leq C_{\mathrm{reg}}  ( \sigma^{1/2} \|\bomega-\bomega_h\|_{0,\O} + \|p-p_h\|_{0,\O}).
\end{equation}

Next, we proceed to test the adjoint problem \eqref{estima2}
against $(\btheta,q):=(\bomega-\bomega_h,p-p_h)$ and to use the
error equation \eqref{errorequation} with
$(\btheta_h, q_h)=(\CR_h \widetilde{\bomega},\Pi_h \widetilde{p})\in \mV_h$
to obtain that
\begin{align*}
& \sigma \;\|\bomega-\bomega_h\|^2_{0,\O} + \|p-p_h\|_{0,\O}^2 =\mathcal{A}(( \bomega-\bomega_h,p-p_h),(\widetilde{\bomega},\widetilde{p}))\\
&=\mathcal{A}((\bomega-\bomega_h,p-p_h),(\widetilde{\bomega}-\btheta_h,\widetilde{p}-q_h))\\
&\leq \Big(1+\sigma^{1/2} + \sigma^{-1/2}\; \nu^{-1/2}   \Vert \boldsymbol\beta \Vert_{\infty,\O}\Big) \;  \Vert (\bomega-\bomega_h, p-p_h) \Vert_{\mV}\Vert (\widetilde{\bomega}-\btheta_h, \widetilde{p}-q_h) \Vert_{\mV}  \\
&\leq   C_{\mathrm{reg}}C_{\mathrm{apx}} h^{\delta} \Big(1+\sigma^{1/2} +  (\sigma\nu)^{-1/2}  \Vert \boldsymbol\beta \Vert_{\infty,\O}\Big)     
\Vert(\bomega-\bomega_h, p-p_h) \Vert_{\mV}  \Big(\|\bomega-\bomega_h\|^2_{0,\O} + \|p-
p_h\|^2_{0,\O}\Big)^{1/2}.
\end{align*}
Here,  we have used  \eqref{prop2N} and \eqref{prop2pi}
with $s=\delta$ and  this completes  the rest of the proof.
\end{proof}

An error estimate in the energy norm can also be derived 
in the following manner.
\begin{theorem}\label{conver}
Assume 
that problem \eqref{WVF3} has a unique solution $(\bomega,p)$ satisfying
the additional regularity
$\bomega\in\H^{s}(\curl;\O)$ and $p\in\H^{1+s}(\O)$,
for some $s\in(1/2,k]$. Then, there exists $C>0$,
independent of $h$, such that the following error estimates hold for $h$ small enough:
\begin{align*}
\Vert(\bomega-\bomega_h, p-p_h)\Vert_{\mV}
& \leq C\,h^{s} \left(\Vert\bomega\Vert_{\H^{s}(\curl;\O)}+\Vert p\Vert_{\H^{1+s}(\O)}\right),\\
\|\bomega-\bomega_h\|_{0,\O} + \|p-p_h\|_{0,\O}   & \leq C \, h^{s+\delta} \left(\Vert\bomega\Vert_{\H^{s}(\curl;\O)}+\Vert p\Vert_{\H^{1+s}(\O)}\right),\end{align*}
where $(\bomega_h,p_h)\in\mV_h$ is the unique
solution to  \eqref{WVF3dis}  with $s=\delta,$ if $s\in(1/2,1)$  and $\delta=1,$ if $s\in [1,k].$
\end{theorem}
\begin{proof}
Now rewrite  \eqref{errorequation}, then use boundedness,  (\ref{prop2N}) and (\ref{prop2pi}) to arrive  at
\begin{align}
&\mathcal{A}((\bomega-\bomega_h,p-p_h),(\bomega-\bomega_h,p-p_h))= \mathcal{A}((\bomega-\bomega_h,p-p_h),(\bomega-\CR_{h}\bomega,p-\Pi_h p))
\nonumber\\
&\qquad \qquad\leq  C\; \Vert(\bomega-\bomega_h,p-p_h)\Vert_{\mV}\; \Vert(\bomega-\CR_{h}\bomega,p-\Pi_h p)\Vert_{\mV}\label{error-1}\\
&\qquad \qquad \leq C\,h^s \Vert(\bomega-\bomega_h,p-p_h)\Vert_{\mV} \left(\Vert\bomega\Vert_{\H^{s}(\curl;\O)}+\Vert p\Vert_{\H^{1+s}(\O)}\right).  \nonumber
\end{align}
Next, for the term on the left-hand side of \eqref{error-1},  apply \eqref{norma} 
to obtain
\begin{align*}
& \mathcal{A}((\bomega-\bomega_h,p-p_h),(\bomega-\bomega_h,p-p_h)) \\
& \geq 1/2\;\Big(\Vert (\bomega-\bomega_h, p-p_h)\Vert^2_{\mV} - (1+\sigma^{-1}\;\nu^{-1}\;\|\boldsymbol\beta\|_{\infty})\;
(\sigma\Vert \bomega-\bomega_h\Vert_{0,\O}^2 + \Vert p-p_h\Vert_{0,\O}^2) \Big).
\end{align*}
Then, a use of Lemma \ref{converl2} yields 
\begin{align*}
& \mathcal{A}((\bomega-\bomega_h,p-p_h),(\bomega-\bomega_h,p-p_h)) \\
& \geq 1/2\;\Big( 1- C   (1+\sigma^{-1}\;\nu^{-1}\;\|\boldsymbol\beta\|_{\infty})\;h^{2\delta} \Big) \;
 \Vert (\bomega-\bomega_h, p-p_h)\Vert^2_{\mV}.
\end{align*}
Choosing $h$ small, the term within brackets, $ ( 1- C   (1+\sigma^{-1}\;\nu^{-1}\;\|\boldsymbol\beta\|_{\infty})\;h^{2\delta})$, can be made  positive; therefore, concluding  the proof.
\end{proof}

As a consequence of  Theorem~\ref{converl2},  with $e_{\bomega}:= \bomega-\bomega_h $ and $e_{p}:= p-p_h$, we obtain the following inf-sup condition:
There exists $\gamma_0>0,$ independent of $h,$ such that,
\begin{equation}\label{inf-sup-1}
\sup_{(\btheta,q) \in \mV } \;\frac{\mathcal{A}((e_{\bomega}, e_ p),(\btheta, q ))} {\Vert (\btheta,q)\Vert_{\mV} } \geq \gamma_{0}\;
\Vert (e_{\bomega}, e_p)\Vert_{\mV}. 
\end{equation}

\subsection{Convergence of the post-processed velocity}\label{velo}
Let $(\bomega_h,p_h)\in\mV_h$ be 
the unique solution of the discrete problem \eqref{WVF3dis}. Then following \eqref{repu2},  we can recover the 
discrete velocity as the following element-wise discontinuous function for each $T\in\cT_{h}(\O)$:
\begin{equation}\label{repudisc}
\bu_h|_{T}:=\sigma^{-1}\left(\mathcal{P}_h\ff-\nu^{-1/2} \bomega_h \times {\boldsymbol\beta}-(\sqrt{\nu}\curl\bomega_h+\nabla p_h)\right)|_{T},
\end{equation}
where $\mathcal{P}_h:\L^2(\O)^3\to \bU_h$ is the $\L^{2}$-orthogonal projector, with
\begin{equation}\label{discont}
\bU_h:=\{\bv_h\in\L^2(\O)^3:
\bv_h|_T\in\cP_{k-1}(T)^{3}\quad\forall T\in\cT_{h}(\O)\}.
\end{equation}
Consequently, we can state an error estimate
for the post-processed velocity. 
\begin{theorem}\label{th:cvu}
Let $(\bomega,p)\in\mV$ 
be the unique solution of \eqref{WVF3}, and $(\bomega_h,p_h)\in\mV_h$ 
be the unique solution of \eqref{WVF3dis}. Assume that $\bomega\in\H^{s}(\curl;\O)$, $p\in\H^{1+s}(\O)$
and $\ff\in\H^{s}(\O)^{3}$,
for some $s\in(1/2,k]$. Then, there exists a positive constant $C$,
independent of $h$, such that
$$\Vert\bu-\bu_h\Vert_{0,\O}\le
C h^{s}\left(\Vert\ff\Vert_{\H^{s}(\O)}+\Vert\bomega\Vert_{\H^{s}(\curl;\O)}+\Vert p\Vert_{\H^{1+s}(\O)}\right).
$$
\end{theorem}
\begin{proof}
From \eqref{repu2}, \eqref{repudisc}, and triangle inequality, it follows that
\begin{align*}
&\Vert\bu-\bu_h\Vert_{0,\O}\\
&\qquad \le\frac{1}{\sigma}\bigl(\Vert\ff-\mathcal{P}_{h}\ff\Vert_{0,\O}
+\Vert\sqrt{\nu}\curl(\bomega_h-\bomega)-\nabla(p-p_h)\Vert_{0,\O}
+\frac{1}{\sqrt{\nu}}\Vert(\bomega-\bomega_h)\times {\boldsymbol\beta} \Vert_{0,\O}\bigr).
\end{align*}
Then, the result follows from standard estimates
satisfied by $\mathcal{P}_h$, as well as from Theorem~\ref{conver}.
\end{proof}

An issue with the post-process \eqref{repudisc} is that it requires 
numerical differentiation (taking the curl of $\bomega_h$ and the gradient of $p_h$).
A possible way to getting around this problem is to  set 
$$\widetilde{\bU}_h:=\{\bv_h\in\H_0^1(\Omega)^3:
\bv_h|_T\in\cP_{k}(T)^{3}\quad\forall T\in\cT_{h}(\O)\},$$  
and recover the discrete velocity in this space, using   
the discrete versions of  \eqref{eq:constitutive}, \eqref{eq:mass}, and \eqref{eq:bc}. 

This results in finding $\tilde{\bu}_h\in \widetilde{\bU}_h $ such 
that 
\begin{equation}\label{eq:vel3}
\nu \int_\Omega  \curl\tilde{\bu}_h\cdot\curl\bv_h + 
\nu \int_\Omega \vdiv\tilde{\bu}_h\,\vdiv\bv_h
 = \sqrt{\nu}\int_\Omega \bomega_h\cdot\curl \bv_h\qquad \forall \bv_h\in\widetilde{\bU}_h.
\end{equation} 
The discrete velocity produced by \eqref{eq:vel3} gives not only 
$$\|\sqrt{\nu}\curl(\bu-\tilde{\bu}_h)\|_{0,\Omega} + \|\sqrt{\nu}\vdiv(\bu-\tilde{\bu}_h)\|_{0,\Omega} = \mathcal{O}(h^s),$$
but also, thanks to the identity relating vector Laplacians with curl and divergence $-\Delta \Phi = \curl\curl \Phi - \nabla(\vdiv\Phi)$, 
one can show, using duality arguments, that 
$$\|\bu-\tilde{\bu}_h\|_{0,\Omega} = \mathcal{O}(h^{s+\delta}),$$
where  $s$ and $\delta$ are given as in Theorem~\ref{conver}.

\section{{\it A posteriori} error analysis  for the 2D problem}\label{sec:apost}
In this section, we 
 propose a residual-based {\it a posteriori} error estimator. For sake of clarity, we restrict our 
analysis to the two-dimensional case (the extension to 3D can be carried out in a similar fashion). 
Therefore, the functional space $\Z$ considered in the {\it a priori} error  analysis now becomes
$\Z:=\H^{1}(\O)$, 
and
\begin{equation}\label{esp1}
\Z_h:=\{\theta_h\in\Z: \theta_h|_T\in\cP_k(T)\quad\forall T\in\cT_{h}(\O)\}.
\end{equation}
We note that in the 2D case, the duality arguments presented
in Section~\ref{sec:FE}, hold  for any $\delta\in(0,1]$.
In particular, this fact will be considered in the definition
of the local {\it a posteriori} error indicator.
Moreover, to keep the notation clear, in this section
we will denote by $\CR_{h}$ the usual Lagrange interpolant in $\Z_h$.

For each $T\in\CT_h$ we let $\CE(T)$ be the set of edges of $T$, and we denote by $\CE_h$
the set of all edges in $\CT_h$, that is 
$$\CE_h=\CE_h(\O)\cup\CE_h(\G),$$
where $\CE_h(\O):=\{e\in\CE_h: e\subset\O\}$, and
$\CE_h(\G):=\{e\in\CE_h: e\subset\G\}$.
In what follows, $h_e$ stands for the diameter of a given edge
$e\in\CE_h$, $\bt_{e}=(-n_2,n_1)$, where $\bn_{e}=(n_1,n_2)$
is a fix unit normal vector of $e$.
Now, let $q\in\LO$ such that $q|_{T}\in C(T)$ for each $T\in\CT_h$,
then, given $e\in\CE_h(\O)$, we denote by $[q]$ the jump of $q$ across $e$,
that is $[q]:=(q|_{T'})|_{e}-(q|_{T''})|_{e}$, where $T'$ and $T''$ are the triangles
of $\CT_{h}$ sharing the edge $e$. Moreover,
let $\bv\in\LO^2$ such that $\bv|_{T}\in C(T)^2$ for each $T\in\CT_h$.
Then, given $e\in\CE_h(\O)$, we denote by $[\bv\cdot\bt]$ the tangential
jump of $\bv$ across $e$,
that is, $[\bv\cdot\bt]:=\left((\bv|_{T'})|_{e}-(\bv|_{T''})|_{e}\right)\cdot\bt_e$,
where $T'$ and $T''$ are the triangles
of $\CT_{h}$ sharing the edge $e$.

Next, let $k\ge 1$ be an integer and let $\Z_h,\Q_h$ and $\bU_h$
be given by \eqref{esp1}, \eqref{space2}, and \eqref{discont}, respectively.
Let $(\omega,p)\in\Z\times\Q$ and
$(\omega_h,p_h)\in\Z_h\times\Q_h$ be the unique solutions
to the continuous and discrete problems \eqref{WVF3} and
\eqref{WVF3dis} with data satisfying $\ff\in\LO^2$ and $\ff\in\H^1(T)^2$ for each $T\in\CT_h$.
We introduce for each $T\in\CT_h$ the local {\it a posteriori} error indicator for $\delta \in (0,1]$ as
\begin{align*}
\widetilde{\betta}_{T}^2&:=h_{T}^{2(1+\delta)}\Vert\rot({\sqrt{\nu}\curl \omega_{h}+\nu^{-1/2} \omega_h \times {\boldsymbol\beta}-\ff})- \nu^{-1/2}\;\sigma\omega_h\Vert_{0,T}^2 \\
&\quad +h_{T}^{2(1+\delta)}\Vert\vdiv(\ff-\nu^{-1/2} \omega_h \times {\boldsymbol\beta}-\nabla p_h)\Vert_{0,T}^2  +\sum_{e\in\CE(T)}\!\!\!h_{e}^{(1+2\delta)}\Vert[({\sqrt{\nu}\curl \omega_{h}+\nu^{-1/2} \omega_h \times{\boldsymbol\beta}-\ff})\cdot\bt]\Vert_{0,e}^{2}\\
&\quad  +\sum_{e\in\CE(T)}\!\!\!h_{e}^{(1+2\delta)}\Vert[(\ff-\nu^{-1/2} \omega_h \times {\boldsymbol\beta}
-\nabla p_h)\cdot\bn]\Vert_{0,e}^{2}\\
&=:h_{T}^{2(1+\delta)}\Big(\Vert \mathcal{R}_1\Vert_{0,T}^2 
+\Vert\mathcal{R}_2\Vert_{0,T}^2 \Big) +\sum_{e\in\CE(T)}\!\!\!h_{e}^{(1+2\delta)} \Big( \Vert[\mathcal{J}_{h,1}\cdot\bt]\Vert_{0,e}^{2}
+ \Vert[\mathcal{J}_{h,2}\cdot\bn]\Vert_{0,e}^{2}\Big),
\end{align*}
and define its global counterpart as 
\begin{equation}\label{globalestimator2}
\tilde\betta:=\left\{\sum_{T\in\CT_h}\tilde\betta_{T}^2\right\}^{1/2}.
\end{equation}
Let us now establish reliability and quasi-efficiency of \eqref{globalestimator2}. 

\subsection{Reliability}
This subsection focuses on proving the reliability 
of the estimator 
in the $\L^2$-norm, and we note that this bound holds for $\delta \in (0,1]$. 
\begin{theorem}\label{th:reliability}
There exists a positive constant $C_{\mathrm{rel}}$,
independent of the discretisation parameter $h$, such that 
\begin{equation}\label{rel-L2-norm}
\Vert \sigma^{1/2} (\omega-\omega_h )\Vert_{0,\Omega}
+ \Vert p-p_h\Vert_{0,\Omega}  \leq {C}_{\mathrm{rel}} \;\widetilde{\betta}.\end{equation}

\end{theorem}
\begin{proof}
Note that
\begin{equation}\label{reliability-1}
\mathcal{A}((e_{\omega}, e_ p),(\theta, q )) = \mathcal{R}(\theta, q ),
\end{equation}
where the residual operator $\mathcal{R}: \Z \times Q \mapsto \R$  is given by
\begin{align*}
\mathcal{R}(\theta, q )&= \mathcal{F}(\theta,q)- \mathcal{A}((\omega_h,p_h),(\theta,q)) \\
& = \bigl(\ff-(\sqrt{\nu}\curl\omega_h+\nabla p_h)-\nu^{-1/2} (\omega_h\;\times\bbbeta),    (\sqrt{\nu}\curl\theta+\nabla q) \bigr)_{0,\O}- \sigma (\omega_h, \theta)_{0,\O}.
\end{align*}
Integration by parts on this residual yields
\begin{align*}
\mathcal{R}(\theta, q) &= \sum_{T\in \CT_h} ( \rot ( {\sqrt{\nu}\curl \omega_{h}+\nu^{-1/2} \omega_h \times {\boldsymbol\beta}-\ff}) - \nu^{-1/2}\;\sigma \omega_h, \sqrt{\nu}\;\theta)_{0,T}\\
&\quad -  \sum_{e\in\CE_h} \langle (\sqrt{\nu}\curl\omega_h+\nu^{-1/2} \;\omega_h\;\times\bbbeta-\ff)\cdot \bt, \sqrt{\nu}\;\theta \rangle_{0,e}\\
&\quad - \sum_{T\in \CT_h} ( \vdiv ( \ff-\nu^{-1/2} \omega_h \;\times{\boldsymbol\beta}-\nabla p_h),  q )_{0,T} + \sum_{e\in\CE_h} \langle (\ff-\nu^{-1/2} \;\omega_h\;\times\bbbeta-\nabla p_h)\cdot \bn, q \rangle_{0,e}\\
&= \sum_{T\in \CT_h} \Big( ( \mathcal{R}_1, \sqrt{\nu}\;\theta)_{0,T} + ( \mathcal{R}_{2},  q )_{0,T}\Big)  +   \sum_{e\in\CE_h}  \Big(  \langle \mathcal{J}_{h,1}\cdot \bt, \sqrt{\nu}\;\theta \rangle_{0,e}   + \langle \mathcal{J}_{h,2}\cdot \bn, q \rangle_{0,e}\Big).
\end{align*}

For the estimate \eqref{rel-L2-norm}, an appeal to the
Aubin-Nitsche argument, using \eqref{estima2} with $(\theta,q)= (\omega-\omega_h, p-p_h)$ 
and 
$(\theta_h, q_h)=(\CR_h \widetilde{\omega},\Pi_h \widetilde{p})\in \mV_h$, now yields
\begin{align*}
 \sigma \;\|\omega-\omega_h\|^2_{0,\O} + \|p-p_h\|_{0,\O}^2 &=\mathcal{A}(( \omega-\omega_h,p-p_h),(\widetilde{\omega},\widetilde{p}))\\
&=\mathcal{A}((\omega-\omega_h,p-p_h),(\widetilde{\omega}-\theta_h,\widetilde{p}-q_h))\\
&=\mathcal{R}(\widetilde{\omega}-\CR_{h}\widetilde{\omega}, \widetilde{p}-\Pi_h \widetilde{p} ).
\end{align*}
%
Then, we can rewrite the residual as  
\begin{align*}
\mathcal{R}(\widetilde{\omega}-\CR_{h}\widetilde{\omega}, \widetilde{p} -\Pi_h \widetilde{p}) 
&= \sum_{T\in \CT_h} \Big( ( \mathcal{R}_1, \sqrt{\nu}\;(\widetilde{\omega}-\CR_{h}\widetilde{\omega}))_{0,T} 
+ \sum_{e\in\CE(T)} \langle \mathcal{J}_{h,1}\cdot \bt, \sqrt{\nu}\;(\widetilde{\omega}-\CR_{h}\widetilde{\omega}) \rangle_{0,e}\Big)\nonumber\\
&\quad + \sum_{T\in \CT_h} \Big(( \mathcal{R}_{2},  \widetilde{p} -\Pi_h \widetilde{p} )_{0,T} +\sum_{e\in\CE(T)} \langle\mathcal{J}_{h,2}\cdot \bn, ( \widetilde{p} -\Pi_h \widetilde{p}) \rangle_{0,e}\Big),
\end{align*}
and an application of the Cauchy-Schwarz inequality together with the approximation properties \eqref{prop2N}, \eqref{prop2pi} and \eqref{elliptic-regularity}  completes the rest of the proof.
\end{proof}

\subsection{Efficiency} This subsection deals with the efficiency of the  {\it a posteriori} error estimator in the weighted $\mV$-norm depending on $\delta\in (0,1)$ (a result that we call quasi-efficiency), and a bound in the $\L^2$-norm, valid for $\delta=1$. 
\begin{theorem} [Quasi-efficiency]\label{thm:quasi-efficiency}
There is a positive constant $C_{\mathrm{eff}}$, independent of $h$, such that for $\delta \in (0,1]$
$$
C_{\mathrm{eff}} \;\widetilde{\betta} \leq C \;\Vert  h^{\delta}_{\cT}  (e_{\omega}, e_p)\Vert_{\mV} + \mathrm{ h.o.t.},
$$
where  $\mathrm{ h.o.t.}$  denotes higher-order terms and $\Vert  h^{\delta}_{\cT_h}  (e_{\omega}, e_p)\Vert_{\mV} := \Big( \sum_{T\in\cT_h} \Vert  h^{\delta}_{T}  (e_{\omega}, e_p)\Vert^2_{\mV(T)}\Big)^{1/2}.$ 
\end{theorem}
The second efficiency result is stated as follows.
\begin{theorem} [Efficiency]\label{thm:efficiency}
There is a positive constant $C_{\mathrm{eff}}$, independent of $h$, such that for $\delta=1$
$$
C_{\mathrm{eff}} \;\widetilde{\betta} \leq C \;\Vert  (\sigma^{1/2} e_{\omega}, e_p)\Vert_{0,\Omega} + \mathrm{ h.o.t.}.
$$
\end{theorem}

A major role in the proof of efficiency is played by element and edge bubbles (locally supported non-negative functions), whose definition we recall in what follows. 
For $T\in\cT_{h}(\O)$ and $e\in \cE(T)$, let $\psi_T$ and $\psi_e$, respectively, be the interior and edge bubble functions defined as in, e.g., \cite{ain-ode}. Let $\psi_T\in \cP_3(T)$ with $\supp (\psi_T) \subset T,$  $\psi_T=0$ on $\partial T$ and
$0\leq \psi_T \leq 1$ in $T.$ Moreover,  let $\psi_e |_T \in \cP_2(T)$ with  $\supp (\psi_e) \subset \O_e := \{T' \in \cT_h(\O):
e\in \cE(T')\},$ $\psi_e =0$ on $\partial T\setminus  e,$ and $0\leq \psi_e \leq 1 $ in $\O_e.$  Again, let us recall an extension  operator $E: C^0(e) \mapsto C^0(T)$ that satisfies $E(q)\in \cP_k(T)$ and $E(q)|_e = q$ for all $q\in \cP_k(e)$ and for all 
$k\in \N\cup \{0\}.$

We now summarise  the properties of $\psi_T, \psi_e$ and $E$. For a proof, see \cite{ain-ode}
or \cite{rv-1996}.
\begin{lemma}\label{lem:psi}
The following properties hold:
\begin{itemize}
\item [(i)] For $T\in \cT_h$ and for $v\in \cP_k(T)$, there is a positive constant $C_1$  such that 
\begin{align*}
C_1^{-1}\; \|v\|^2_{0,T} &\leq \int_{T} \psi_T v^2 \dx \leq C_1 \|v\|^2_{0,T},\\
C_1^{-1}\; \|v\|^2_{0,T} &\leq \|\psi v\|^2_{0,T} + h_T^2  |\psi v|^2_{1,T} \leq C_1 \|v\|^2_{0,T}. 
\end{align*}
\item [(ii)] For $e\in \cE_h$ and $v\in \cP_k(e)$, there  exists a positive constant say $C_1$ such that
$$C_1^{-1}\; \|v\|^2_{0,e} \leq \int_{e} \psi_e v^2 \ds \leq C_1 \|v\|^2_{0,e} .$$
\item [(iii)] For $T\in \cT_h$ with $e\in \cE(T)$ and  for all $v\in \cP_k(e)$, there is a positive constant again say $C_1$
such that
$$ \| \psi_e^{1/2}  \;E(v) \|^2_{0,T}  \leq C_1 h_e\; \|v\|^2_{0,e} .$$
\end{itemize}
\end{lemma}

\vspace{2em}
\noindent{\it { Proof of Theorem \ref{thm:quasi-efficiency}}}.  
With the help of the $\L^2(T)^2$-orthogonal projection $\cP_{T}^{\ell}$ onto $\cP_{\ell}(T)^2$, for $\ell\geq k,$ with respect to the weighted $\L^2$-inner product $(\psi_T \ff,\bg)$, for 
$\ff,\bg\in \L^2(T)^2,$ it now follows that
\begin{align*}
\|\mathcal{R}_1\|_{0,T}^2 &= \|\nu^{-1/2} \sigma \omega_h+ \rot({\sqrt{\nu}\curl \omega_{h}+\nu^{-1/2} \omega_h \times\;{\boldsymbol\beta}-\ff})\|^2_{0,T}\\
&\leq  2 \Big(\|\rot{\ff}- \cP_{T}^{\ell} (\rot{\ff})\|^2_{0,T} + \|\cP_T^{\ell} (\nu^{-1/2}  \sigma \omega_h + \rot({\sqrt{\nu}\curl \omega_{h}+\nu^{-1/2} \omega_h\;\times {\boldsymbol\beta}-\ff}))\|^2_{0,T}\Big).
\end{align*} 
For the second term on the right-hand side, a use of Lemma~\ref{lem:psi} shows that 
\begin{align*}
& \|\cP_T^{\ell} (\nu^{-1/2}\;\sigma \omega_h + \rot({\sqrt{\nu}\curl \omega_{h}+\nu^{-1/2} \omega_h\;\times {\boldsymbol\beta}-\ff}))\|^2_{0,T}= \Vert \cP_T^{\ell} \mathcal{R}_1\Vert_{0,T}^2 \\
&\qquad \leq \Vert \psi_T^{1/2} \cP_T^{\ell} \mathcal{R}_1\Vert^2_{0,T}
= (\psi_T \cP_T^{\ell}\;\mathcal{R}_1, \mathcal{R}_1)_{0,T}. 
\end{align*} 
In a similar manner, we can derive the bounds
$$\|\mathcal{R}_2\|^2_{0,T} \leq 2 \Big(\|\vdiv \ff  - \cP^{\ell}_{T} (\vdiv \ff)\|_{0,T}^2+ \| \cP^{\ell}_T\mathcal{R}_2\|^2_{0,T}\Big),$$
and 
$$\|\cP^{\ell}_T\mathcal{R}_2\|^2_{0,T} \leq \|\psi_T^{1/2} \cP_T^{\ell}\;\mathcal{R}_2\|^2_{0,T}= (\psi_T  \cP_T^{\ell}\;\mathcal{R}_2, \mathcal{R}_2)_{0,T}.
$$
We proceed to choose $(\theta, q)= \psi_T (\cP_T^{\ell}\;\mathcal{R}_1,\cP_T^{\ell}\;\mathcal{R}_2)$ in 
\eqref{reliability-1}  and  obtain
\begin{align}
\|\psi^{1/2}_T\cP_T^{\ell}\; \mathcal{R}_1\|^2_{0,T} +\| \psi^{1/2}_T \cP_T^{\ell}\;\mathcal{R}_2 \|^2_{0,T} &= ((\mathcal{R}_1,\mathcal{R}_2), 
\psi_T(\cP_T^{\ell}\;\mathcal{R}_1,\cP_T^{\ell}\;\mathcal{R}_2) )_{0,T}  \nonumber \\
& =
\mathcal{A}_{T} ((e_{\omega}, e_ p), \psi_T(\cP_T^{\ell}\;\mathcal{R}_1, \cP_T^{\ell}\;\mathcal{R}_2)).\label{estimator-V-norm}
\end{align}

Next, we invoke estimate (i) of Lemma~\ref{lem:psi}. This  yields
\begin{align*}
\|\psi^{1/2}_T\cP_T^{\ell}\; \mathcal{R}_1\|^2_{0,T} +\| \psi^{1/2}_T \cP_T^{\ell}\;\mathcal{R}_2 \|^2_{0,T} 
&\leq C\, \|(e_{\omega}, e_ p)\|_{\mV(T)}\; \|\psi_T \cP_T^{\ell}\;(\mathcal{R}_1, \mathcal{R}_2)\|_{\mV(T)}\nonumber\\
&\leq C\,h_T^{-1}  \|(e_{\omega}, e_ p)\|_{\mV(T)}\;\Big( \|\mathcal{R}_1\|^2_{0,T} +\| \mathcal{R}_2)\|^2_{0,T}\Big)^{1/2}.
\end{align*}
Altogether, we now arrive at 
\begin{align}\label{estimate:R1-R2}
h_T^{2(1+\delta)}\;\Big( \| \mathcal{R}_1\|^2_{0,T} +\| \mathcal{R}_2 \|^2_{0,T} \Big)
&\leq  C\,\Big( \|h_T^{\delta} (e_{\omega}, e_ p)\|_{\mV(T)}^2 + h_T^{2(1+\delta)}  (  \|\rot \ff- \cP^{\ell}_T (\rot \ff)\|^2_{0,T} \nonumber\\
& \qquad \qquad \qquad \qquad \qquad \qquad +  \|\vdiv \ff- \cP^{\ell}_T (\vdiv \ff)\|^2_{0,T}) \Big).
\end{align}
Regarding the estimates associated with $\mathcal{J}_{h,1}$ and $\mathcal{J}_{h,2}$, we introduce, respectively,   $\widetilde{\cP}_T^{\ell}$  and  $\widetilde{\cP}_e^{\ell}$ as  the weighted $\L^2$-orthogonal projections (say, with respect to the weighted inner product $(\psi_e f,g)_e$), onto $\cP_{\ell}(T)^2$ and $\cP_{\ell}(e)$, for $\ell\geq k.$  Then, we can  bound $\mathcal{J}_{h,1}$  and  $\mathcal{J}_{h,2}$  as
\begin{align}\label{estimate:J1-J2}
h_{e}^{(1+2\delta)}\Big(\| [\mathcal{J}_{h,1}\cdot \bt] \|_{0,e}^2 + \| [\mathcal{J}_{h,2}\cdot \bn] \|_{0,e}^2\Big) & \leq h_{e}^{(1+2\delta)}
\Big(  \| [( \ff- \widetilde{\cP}_T^{\ell} \ff)\cdot\bt] \|_{0,e}^2 +  \| [( \ff- \widetilde{\cP}_T^{\ell} \ff)\cdot\bn] \|_{0,e}^2\Big)\nonumber\\
&\quad + h_{e}^{(1+2\delta)}  \Big( \| [\widetilde{\cP}_e^{\ell} (\mathcal{J}_{h,1})\cdot\bt]\|_{0,e}^2 + \| [\widetilde{\cP}_e^{\ell}(\mathcal{J}_{h,2})\cdot \bn]\|_{0,e}^2\Big).
\end{align}
In order to estimate the first term on the right-hand side of \eqref{estimate:J1-J2} we use the trace inequality, yielding 
\begin{align}\label{estimate:ff}
h_{e}^{(1+2\delta)} \Big(\| [ (\ff- \widetilde{\cP}_e^{\ell} \ff)\cdot \bt] \|_{0,e}^2 &+ \| [ (\ff- \widetilde{\cP}_e^{\ell} \ff)\cdot \bn] \|_{0,e}^2\Big) \nonumber\\
& \leq C\; h_{e}^{(1+2\delta)} \sum_{T\in\Omega_{e}} \Big( h_{e}^{-1} \| \ff- \widetilde{\cP}_e^{\ell} \ff \|_{0,T}^2 + h_e  \| \nabla(\ff-\widetilde{\cP}_e^{\ell}\ff) \|_{0,T}^2\Big)\nonumber\\
&\leq C\;\sum_{T\in\Omega_{e}} \Big( h_{T}^{2\delta} \| \ff- \widetilde{\cP}_e^{\ell} \ff \|_{0,T}^2 + h_{T}^{2(1+\delta)}  \| \nabla(\ff-\widetilde{\cP}_e^{\ell}\ff) \|_{0,T}^2\Big).
\end{align}
Again from \eqref{reliability-1} we note that with $(\btheta, q)= \psi_e E ([\widetilde{\cP}_e^{\ell}\mathcal{J}_{h,1}\cdot\bt], [\widetilde{\cP}_e^{\ell}\mathcal{J}_{h,2}\cdot\bn] )_e$ we obtain 
\begin{align*}
\mathcal{A}_{T} ((e_{\omega}, e_ p), \psi_e E ([\widetilde{\cP}_e^{\ell}\mathcal{J}_{h,1}\cdot\bt],  [\widetilde{\cP}_e^{\ell}\mathcal{J}_{h,2}\cdot\bn] )_{0,e}) & = ((\mathcal{R}_1,\mathcal{R}_2),
\psi_e E ([\widetilde{\cP}_e^{\ell}\mathcal{J}_{h,1}\cdot\bt], [\widetilde{\cP}_e^{\ell}\mathcal{J}_{h,2}\cdot\bn] ))_{0,\O_e} \nonumber\\
&\quad + ( ([\mathcal{J}_{h,1}\cdot\bt], [\mathcal{J}_{h,2}\cdot\bn] ), \psi_e E ([\widetilde{\cP}_e^{\ell}\mathcal{J}_{h,1}\cdot\bt], [\widetilde{\cP}_e^{\ell}\mathcal{J}_{h,2}\cdot\bn] ))_{0,e}.
\end{align*} 
Now we appeal again to Lemma~\ref{lem:psi} to readily find that
\begin{equation*}
( ([\mathcal{J}_{h,1}\cdot\bt], [\mathcal{J}_{h,2}\cdot\bn] ), \psi_e E ([\widetilde{\cP}_e^{\ell}\mathcal{J}_{h,1}\cdot\bt], [\widetilde{\cP}_e^{\ell}\mathcal{J}_{h,2}\cdot\bn] ))_{0,e} \geq C_1\;\Big( \| [\widetilde{\cP}_e^{\ell}\mathcal{J}_{h,1}\cdot\bt] \|_{0,e}^2 + \| [\widetilde{\cP}_e^{\ell}\mathcal{J}_{h,2}\cdot\bn] \|_{0,e}^2\Big),
\end{equation*}
and, thus, we arrive at 
\begin{align*}
\Big(\| [\widetilde{\cP}_e^{\ell}\mathcal{J}_{h,1}\cdot\bt] \|_{0,e}^2 + \| [\widetilde{\cP}_e^{\ell}\mathcal{J}_{h,2}\cdot\bn] \|_{0,e}^2\Big)
&\leq  C_1^{-1} \;\Big( |\mathcal{A}_{T} ((e_{\omega}, e_ p), \psi_e E ([\widetilde{\cP}_e^{\ell}\mathcal{J}_{h,1}\cdot\bt], [\widetilde{\cP}_e^{\ell}\mathcal{J}_{h,2}\cdot\bn] )_{0,e}) | \nonumber\\
&\qquad \qquad + | ((\mathcal{R}_1,\mathcal{R}_2), \psi_e E ([\widetilde{\cP}_e^{\ell}\mathcal{J}_{h,1}\cdot\bt], [\widetilde{\cP}_e^{\ell}\mathcal{J}_{h,2}\cdot\bn] ))_{0,\O_e} |\Big).
\end{align*}
Therefore, employing properties (i) and (ii) from Lemma~\ref{lem:psi}, it follows that
\begin{align*}
\Big(\| [\widetilde{\cP}_e^{\ell}\mathcal{J}_{h,1}\cdot\bt] \|_{0,e}^2 &+ \| [\widetilde{\cP}_e^{\ell}\mathcal{J}_{h,2}\cdot\bn] \|_{0,e}^2\Big)
\leq C\, \Big( \|(e_{\omega}, e_ p)\|_{\mV(\O_e)} \: \| \psi_e^{1/2} E ([\widetilde{\cP}_e^{\ell}\mathcal{J}_{h,1}\cdot\bt], [\widetilde{\cP}_e^{\ell}\mathcal{J}_{h,2}\cdot\bn] )\|_{\mV(\O_e)} \\
&\qquad \qquad \qquad \qquad \qquad \qquad + \|(\mathcal{R}_1,\mathcal{R}_2)\|_{0,\O_e}\;\| \psi_e^{1/2} E ([\widetilde{\cP}_e^{\ell}\mathcal{J}_{h,1}\cdot\bt], [\widetilde{\cP}_e^{\ell}\mathcal{J}_{h,2}\cdot\bn] )\|_{0,\O_e}\Big)\\
&\leq C\,\Big( h_T^{-1} h^{1/2}_e \|(e_{\omega}, e_ p)\|_{\mV(\O_e)}  + h^{1/2}_e \:\|(\mathcal{R}_1,\mathcal{R}_2)\|_{0,\O_e}\Big)\; 
\Big(\| [\widetilde{\cP}_e^{\ell}\mathcal{J}_{h,1}\cdot\bt] \|_{0,e}^2 + \| [\widetilde{\cP}_e^{\ell}\mathcal{J}_{h,2}\cdot\bn] \|_{0,e}^2 \Big)^{1/2}.
\end{align*}
Now with $h_e \leq h_T$, we simply apply  \eqref{estimate:R1-R2} and obtain
\begin{align}\label{estimate:J1-J2-2}
& h_e^{\frac{1}{2}+\delta}\;\Big(\| [\widetilde{\cP}_e^{\ell}\mathcal{J}_{h,1}\cdot\bt] \|_{0,e}^2 + \| [\widetilde{\cP}_e^{\ell}\mathcal{J}_{h,2}\cdot\bn] \|_{0,e}^2\Big)^{1/2}
\leq C\Big( h_e^{\delta} \|(e_{\omega}, e_ p)\|_{\mV(\O_e)}  + h_e^{1+\delta} \:\|(\mathcal{R}_1,\mathcal{R}_2)\|_{0,\O_e}\Big)\nonumber\\
 &\qquad \leq C\Big( h_T^{2\delta} \|(e_{\omega}, e_ p)\|^2_{\mV(\O_e)}  +  h_T^{2(1+\delta)}(  \|\rot \ff- \cP^{\ell}_T (\rot \ff)\|^2_{0,T}  + \|\vdiv \ff- \cP^{\ell}_T (\vdiv \ff)\|^2_{0,T}) \Big)^{1/2}.
\end{align}
Finally, we substitute \eqref{estimate:ff} and \eqref{estimate:J1-J2-2} in \eqref{estimate:J1-J2}, and then combine the result with  \eqref{estimate:R1-R2} to complete the rest of the proof.
\hfill {$\Box$}

\vspace{2em}
\noindent{\it { Proof of Theorem \ref{thm:efficiency}}}.  
We follow the same steps taken in the proof of  Theorem~\ref{thm:quasi-efficiency} until arriving to relation \eqref{estimator-V-norm}. Then, applying integration by parts and exploiting the properties of $\psi_T$ 
we can show
\begin{align}\label{estimate:L2-norm}
\|\psi^{1/2}_T\cP_T^{\ell}\; \mathcal{R}_1\|^2_{0,T} &+\| \psi^{1/2}_T \cP_T^{\ell}\;\mathcal{R}_2 \|^2_{0,T} = 
 (\sigma^{1/2} e_{\omega},  \sigma^{1/2} \psi_T(\cP_T^{\ell}\;\mathcal{R}_1))_{0,T}\nonumber\\
 &+ ( \sigma^{1/2} e_{\omega} \times \bbbeta, \sigma^{-1/2} \nu^{-1/2} ( \curl ( \psi_T \cP_T^{\ell}\;\mathcal{R}_1) + \nabla(\psi_T \cP_T^{\ell}\;\mathcal{R}_2)) \;)_{0,T}\nonumber\\
 &-(\sigma^{1/2} e_{\omega}, \sigma^{-1/2} (\nu \curl (\curl(\psi_T \cP_T^{\ell}\;\mathcal{R}_1)) )_{0,T} - (e_p, \Delta(\psi_T\cP_T^{\ell}\;\mathcal{R}_2) )_{0,T}.
\end{align}
An application of estimate (i) of Lemma~\ref{lem:psi} together with inverse inequality implies that 
\begin{align*}
\|\psi^{1/2}_T\cP_T^{\ell}\; \mathcal{R}_1\|^2_{0,T} +\| \psi^{1/2}_T \cP_T^{\ell}\;\mathcal{R}_2 \|^2_{0,T} 
&\leq C\,h_T^{-2} \Big( \Vert \sigma^{1/2} e_{\omega}\Vert_{0,T} + \Vert e_ p \Vert_{0,T}\Big)\; \Big(\Vert \psi_T \cP_T^{\ell}\;\mathcal{R}_1\Vert_{0,T} + \Vert \psi_T \cP_T^{\ell}\; \mathcal{R}_2 \Vert_{0,T}\Big)\nonumber\\
&\leq C\,h_T^{-2}  \|(\sigma^{1/2} e_{\omega}, e_ p)\|_{0,T}\;\Big( \|\mathcal{R}_1\|^2_{0,T} +\| \mathcal{R}_2\|^2_{0,T}\Big)^{1/2}.
\end{align*}
Altogether, we now obtain
\begin{align}\label{estimate:R1-R2-1}
h_T^{4}\;\Big( \| \mathcal{R}_1\|^2_{0,T} +\| \mathcal{R}_2 \|^2_{0,T} \Big)
&\leq  C\,\Big( \| (\sigma^{1/2} e_{\omega}, e_ p)\|_{0,T}^2 + h_T^{4}  (  \|\rot \ff- \cP^{\ell}_T (\rot \ff)\|^2_{0,T} \nonumber\\
& \qquad \qquad \qquad \qquad \qquad \qquad +  \|\vdiv \ff- \cP^{\ell}_T (\vdiv \ff)\|^2_{0,T}) \Big).
\end{align}
For the estimates of $\mathcal{J}_{h,1}$ and $\mathcal{J}_{h,2}$,  we again proceed as in the proof of Theorem ~\ref{thm:quasi-efficiency}  to arrive at \eqref{estimate:J1-J2}.
Then, an integration by parts applied to the first term on the right-hand side of \eqref{estimate:J1-J2} as in \eqref{estimate:L2-norm}, with estimates (i) and (ii) from  Lemma~\ref{lem:psi}, in combination with 
inverse inequality, and obvious cancellation, permit us to write 
\begin{align*}
\Big(\| [\widetilde{\cP}_e^{\ell}\mathcal{J}_{h,1}\cdot\bt] \|_{0,e}^2 &+ \| [\widetilde{\cP}_e^{\ell}\mathcal{J}_{h,2}\cdot\bn] \|_{0,e}^2\Big)^{1/2}
\leq C\,\Big( h_T^{-2} \|(\sigma^{1/2} e_{\omega}, e_ p)\|_{0,\O_e}  + h^{1/2}_{T} \:\|(\mathcal{R}_1,\mathcal{R}_2)\|_{0,\O_e}\Big).
\end{align*}
Since $h_e \leq h_T$, we simply apply  \eqref{estimate:R1-R2-1}  to   obtain, after squaring, the bound
\begin{align}\label{estimate:J1-J2-3}
& h_e^{3}\;\Big(\| [\widetilde{\cP}_e^{\ell}\mathcal{J}_{h,1}\cdot\bt] \|_{0,e}^2 + \| [\widetilde{\cP}_e^{\ell}\mathcal{J}_{h,2}\cdot\bn] \|_{0,e}^2\Big)
\leq C\Big(  \|(\sigma^{1/2} e_{\omega}, e_ p)\|^2_{0,\O_e}  + h_{T}^{4} \:\|(\mathcal{R}_1,\mathcal{R}_2)\|^2_{0,\O_e}\Big)\nonumber\\
&\qquad \leq C\Big(  \|(\sigma^{1/2} e_{\omega}, e_ p)\|^2_{0,\O_e}  +  h_T^{4}(  \|\rot \ff- \cP^{\ell}_T (\rot \ff)\|^2_{0,T}   + \|\vdiv \ff- \cP^{\ell}_T (\vdiv \ff)\|^2_{0,T}) \Big).
\end{align}
On substitution of \eqref{estimate:J1-J2-3} and \eqref{estimate:ff}  in \eqref{estimate:J1-J2} for $\delta=1$, it suffices to combine the resulting estimate with 
\eqref{estimate:R1-R2-1} to conclude the rest of the proof. \hfill {$\Box$}

\begin{remark}
Note that the {\it a posteriori} lower bound derived in Theorem~\ref{thm:efficiency} is valid  only upon the assumption of $\H^2$-regularity, that is,  for $\delta=1$. When $\delta\in (0,1)$, obtaining an efficiency result for the {\it a posteriori} error indicator in the $\L^2$-norm is much more involved, essentially due to the presence of corner singularities.  For instance, a 
reliable and efficient estimators using weighted $\L^2$-norms is available for the Poisson equation in \cite{tpw07}. A similar analysis could eventually be carried out in the present case, provided an additional regularity is established using weighted Sobolev spaces and appropriate interpolation results. However here we restrict ourselves only to verifying these properties numerically in the next Section.  

In addition, the result of Theorem~\ref{thm:quasi-efficiency} does indicate that the estimator is quasi-efficient, as the error in the $\L^2$-norm, $\Vert (\sigma^{1/2} e_{\omega},e_p)\Vert_{0,\O}$, is proportional to
$C\; \Vert(e_{\omega},e_p)\Vert_{0,\O}.$
\end{remark}

\section{Numerical tests}\label{sec:numer}

\begin{table}[t]
\setlength{\tabcolsep}{2pt}
\begin{center}
{\small\begin{tabular}{|c|c|c|c|c|c|c|c|c|c|c|}
  \hline
  $h$& $\Vert\bomega-\bomega_h\!\Vert_{0,\O}$ & \texttt{rate}  & $\Vert p-p_h\!\Vert_{0,\O}$
  & \texttt{rate} &  $\Vert\bu-\bu_h\!\Vert_{0,\O}$ & \texttt{rate} & $\Vert\bu-\tilde{\bu}_h\!\Vert_{0,\O}$ & \texttt{rate} &
  {$\Vert(\bomega,p)\!-\!(\bomega_h,p_h\!)\Vert_{\mV}$} & {\texttt{rate}}\\ \hline \hline
\multicolumn{11}{|c|}{$k = 1$} \\
\hline
1.414 & 5.1820 & --    & 5.661091 & --    & 2.812797 & --    & 2.8105 & --    & 12.9222 & --\\ \hline
0.745 & 1.4824 & 1.954 & 0.601930 & 3.499 & 1.564395 & 0.916 & 2.3300 & 0.420 & 7.54263 & 0.840\\ \hline
0.380 & 0.5602 & 1.445 & 0.222225 & 1.480 & 0.871818 & 0.868 & 0.5504 & 2.143 & 4.99371 & 0.612\\ \hline
0.190 & 0.1222 & 2.196 & 0.047772 & 2.217 & 0.428659 & 1.024 & 0.1257 & 2.129 & 2.26335 & 1.141\\ \hline
0.096 & 0.0278 & 2.175 & 0.008442 & 2.548 & 0.212433 & 1.032 & 0.0321 & 2.005 & 1.10120 & 1.059\\ \hline
0.051 & 0.0074 & 2.082 & 0.002089 & 2.197 & 0.106335 & 1.088 & 0.0080 & 2.186 & 0.55034 & 1.091\\ \hline
0.028 & 0.0018 & 2.297 & 0.000489 & 2.377 & 0.053041 & 1.138 & 0.0019 & 2.282 & 0.27367 & 1.143\\ \hline
0.014 & 0.0004 & 2.184 & 0.000123 & 2.208 & 0.026735 & 1.097 & 0.0005 & 2.193 & 0.13906 & 1.084  \\ \hline
\multicolumn{11}{|c|}{$k = 2$} \\
\hline
1.414 & 1.603335 & --    & 2.180130 & --    & 3.4023 & --    & 2.230100 & --    & 9.7142 & --\\ \hline
0.745 & 0.491516 & 1.846 & 0.195556 & 3.764 & 2.3595 & 1.028 & 0.316770 & 3.047 & 4.6790 & 1.140\\ \hline
0.380 & 0.057665 & 3.182 & 0.016245 & 3.695 & 0.4888 & 2.338 & 0.043145 & 2.961 & 0.8417 & 2.547\\ \hline
0.190 & 0.008520 & 2.758 & 0.001088 & 3.899 & 0.1180 & 2.050 & 0.005448 & 2.985 & 0.1939 & 2.117\\ \hline
0.096 & 0.001220 & 2.857 & 0.000042 & 4.768 & 0.0316 & 1.934 & 0.000629 & 3.175 & 0.0520 & 1.933\\ \hline
0.051 & 0.000155 & 3.241 & 0.000005 & 3.222 & 0.0078 & 2.200 & 0.000081 & 3.233 & 0.0127 & 2.218\\ \hline
0.028 & 0.000020 & 3.380 & 0.000001 & 3.465 & 0.0019 & 2.280 & 0.000010 & 3.401 & 0.0031 & 2.268\\ \hline
0.014 & 0.000002 & 3.375 & 1.34e-07& 3.481 & 0.0004 & 2.203 & 0.000001 & 3.383 & 0.0008 & 2.203  \\ \hline
\end{tabular}}
\end{center}
\caption{Example~1. Convergence tests against analytical solutions 
on a sequence of uniformly refined triangulations of the domain 
$\O=(-1,1)^2$. Approximations with $k = 1,2$ and velocity postprocessing using \eqref{repudisc} 
and \eqref{eq:vel3}.}  \label{table:ex1a}
\end{table}

In this section, we report  the results of some numerical tests carried out
with the finite element method proposed in Section~\ref{sec:FE}. The solution 
of all linear systems is carried out with the multifrontal massively parallel sparse 
direct solver MUMPS. 

The discrete formulation is extended to the case of mixed boundary 
conditions, assuming that the domain boundary is disjointly split 
into two parts $\G_1$ and $\G_2$ such that \eqref{eq:bc} is replaced by
\begin{align}
  \bu  = &\,\bg &    \mbox{ on } \G_1,\nonumber\\
  \bu \times\bn = &\,\ba \times \bn&    \mbox{ on } \G_2,\label{eq:bc2}\\
  p  = &\, p_0 & \mbox{ on } \G_2, \nonumber
 \end{align}
(see similar treatments in \cite{Bertoluzza,bernd}) and the condition of zero average is imposed on the Bernoulli pressure, using a real Lagrange multiplier approach, only if $\G_2=\emptyset$. Using \eqref{eq:bc2}, the linear functional $\mathcal{F}_h:\mV_h\to\R$ defining the finite element scheme 
adopts the specification
$$
\mathcal{F}(\btheta_h,q_h)=\int_{\O}\ff\cdot(\sqrt{\nu}\curl\btheta_h+\nabla q_h)+\sigma \sqrt{\nu} \langle \bg \times \bn, \btheta_h\rangle_{\Gamma_1}
- \sigma \langle \bg \cdot \bn , q_h \rangle_{\Gamma_1} 
+\sigma \sqrt{\nu} \langle \ba \times \bn, \btheta_h\rangle_{\Gamma_2}.
$$

\begin{table}[t]
\setlength{\tabcolsep}{2pt}
\begin{center}
{\small\begin{tabular}{|c|c|c|c|c|c|c|c|c|c|c|c|}
  \hline
DoF &   $\Vert\bu-\bu_h\Vert_{0,\O}$ & \texttt{rate} & $\Vert\bu-\tilde{\bu}_h\Vert_{0,\O}$ & \texttt{rate} &  $\Vert(\sqrt{\sigma} e_{\bomega},e_p) \Vert_{0,\O}$ & \texttt{rate} & $\Vert h^{\delta}_{\cT_h}(e_{\bomega},e_p)\Vert_{\mV}$ & \texttt{rate} & \texttt{eff}$_1(\tilde\betta)$  & \texttt{eff}$_2(\tilde\betta)$  \\ \hline \hline
\multicolumn{11}{|c|}{$\delta = 1/10$} \\
\hline
27 & 7.353-02 &   -- & 0.00371 &  -- & 0.08482 &  -- &  0.89616 &  -- & 0.0281 & 0.2964\\
83 & 3.02e-02 & 1.281 & 0.00096 & 1.947 & 0.02324 &  1.872 &  0.43461 &  1.044 & 0.0142 & 0.2668\\
291 & 1.14e-02 & 1.401 & 0.00023 & 2.056 & 0.00591 &  1.970 &  0.20478 &  1.086 & 0.00722 & 0.2498\\
1091 & 4.18e-03 & 1.453 & 5.66e-05 & 2.032 & 0.00149 &  1.993 &  0.09575 &  1.097 & 0.00374 & 0.2411\\
4227 & 1.50e-03 & 1.477 & 1.42e-05 & 2.015 & 0.00037 &  1.998 &  0.04467 &  1.099 & 0.00197 & 0.2367\\
16643 & 5.35e-04 & 1.493 & 3.48e-06 & 2.007 & 9.32e-05 &     2.001 &  0.02084 &   1.100 & 0.00105 & 0.2345\\
66051 & 1.90e-04 & 1.494 & 8.69e-07 & 2.003 & 2.33e-05 &     2.000 &  0.00972 &   1.100 & 0.000558 & 0.2334\\
263171 & 6.73e-05 & 1.500 & 2.17e-07 & 2.002 & 5.81e-06 &     2.000 &  0.00453 &   1.100 & 0.000298 & 0.2328\\
\hline
 \multicolumn{11}{|c|}{$\delta = 1/2$} \\
 \hline
 27 &   7.35e-02 &   -- & 0.00371 &   -- &  0.08482 &   -- &  0.67910 &  -- & 0.0372 & 0.2398 \\
83 & 3.02e-02 & 1.281 & 0.00096 & 1.958 &  0.02324 &  1.872 &  0.25282 & 1.449 & 0.0249 & 0.2382\\
291 &  1.14e-02 & 1.401 & 0.00023 & 2.056 &  0.00591 &  1.970 &  0.08912 & 1.492 & 0.0167 & 0.2395\\
1091 &  4.18e-03 & 1.453 & 5.66e-05 & 2.031 &  0.00149 &  1.993 &  0.03176 & 1.500 & 0.0114 & 0.2395\\
4227 &  1.50e-03 & 1.477 & 1.42e-05 & 2.025 &  0.00037 &  1.998 &  0.01152 & 1.500 & 0.0079 & 0.2395\\
16643 & 5.35e-04 & 1.493 & 3.48e-06 & 2.013 &  9.32e-05 &     2.001 &  0.00395 & 1.500 & 0.0055 & 0.2395\\
66051 &  1.90e-04 & 1.494 & 8.69e-07 &   2.000 &  2.33e-05 &     2.000 &  0.00147 & 1.500 & 0.0039 & 0.2395\\
263171 &  6.73e-05 & 1.500 & 2.17e-07 &   2.000 &  5.81e-06 &     2.000 &  0.00049 & 1.500 & 0.0027 & 0.2395\\
\hline
 \multicolumn{11}{|c|}{$\delta = 1$} \\
 \hline
 27 & 7.35e-02 & -- & 0.00371 &   -- & 0.0848 &    -- &  0.48022 &    -- & 0.0452 & 0.2390 \\
83 & 3.02e-02 & 1.281 & 0.00096 & 1.958 & 0.0232 &  1.872 &  0.12480 &  1.942 & 0.0450 & 0.2397\\
291 & 1.14e-02 & 1.401 & 0.00023 & 2.056 & 0.00591 &  1.970 &  0.03152 &  1.991 & 0.0448 & 0.2395\\
1091 & 4.18e-03 & 1.453 & 5.66e-05 & 2.031 & 0.00149 &  1.993 &  0.00789 &    2.000 & 0.0448 & 0.2394\\
4227 & 1.50e-03 & 1.477 & 1.42e-05 & 2.025 & 0.00037 &    1.998 &  0.00197 &    2.000 & 0.0449 & 0.2395\\
16643 & 5.35e-04 & 1.493 & 3.48e-06 & 2.013 & 9.3e-05 &    2.001 &  0.00049 &    2.000 & 0.0448 & 0.2395\\
66051 & 1.90e-04 & 1.494 & 8.69e-07 &   2.000 & 2.33e-05 &    2.000 &  0.00012 &    2.000 & 0.0448 & 0.2395\\
263171 & 6.73e-05 & 1.500 & 2.17e-07 &   2.000 & 5.81e-06 &    2.000 &  3.08e-05 &    2.000 & 0.0448 & 0.2395\\
\hline
\end{tabular}}
\end{center}
\caption{Example~2A. Error history and effectivity indexes \eqref{eq:indexes} associated with the {\it a posteriori} error estimator \eqref{globalestimator2}. Smooth solutions on the unit square. 
Approximation with $k = 1$, and velocity postprocessing using \eqref{repudisc} 
and \eqref{eq:vel3}.}  \label{table:aposte}
\end{table}

\medskip\noindent\textbf{Example 1.} 
First, we construct a manufactured solution in the two-dimensional domain 
$\Omega=(-1,1)^2$ and assess the convergence properties and verify the 
rates anticipated in Lemma~\ref{converl2}, and Theorems~\ref{conver} and \ref{th:cvu}. 
We compute individual errors  and convergence rates as usual
for all fields on su\-cce\-ssi\-vely refined
partitions of $\O$. For this test we assume that $\G_1$ is composed by the 
horizontal edges and the right edge, whereas $\G_2$ is the rest of the boundary.
We propose the following closed-form and smooth solutions
\begin{gather*}
\bomega(x,y):=-\sqrt{\nu}\left(e^{x-1}\sin(\pi y)^2
+2\pi^2(x-e^{x-1})(\sin(\pi y)^2-\cos(\pi y)^2)\right),\quad 
p(x,y):=x^4-y^4, \\
\bu(x,y):=\left(
_{}\begin{array}{l}
(e^{x-1}-x)(2\pi\sin(\pi y)\cos(\pi y))\\
-(e^{x-1}-1)(\sin(\pi y)^2) \\
\end{array}
\right),
\end{gather*}
satisfying $\bu=\cero$ on $\Gamma_1$.
\noindent
In addition, we consider 
\[ \bbbeta(x,y):=\left(
_{}\begin{array}{l}
\frac{1}{6}(e^{x-1}-x)(\pi\sin(2\pi y)) \\
-(e^{x-1}-1)(\sin(\pi y)^2) \\
\end{array}
\right),
\]
together with the model parameters $\sigma=100$ and $\nu=0.1$, which in turn  
fulfil \eqref{bound1}. These exact solutions lead to a nonzero right-hand side 
that we use to verify the accuracy of the finite element approximation.

We report in Table \ref{table:ex1a} the error
history of the method in the $\L^2$- and $\mV$-norms, where we also show the 
convergence of the post-processed velocity using the direct computation 
\eqref{repudisc} producing $\bu_h\in \bU_h$, and the alternative post-processing 
through solving the auxiliary problem \eqref{eq:vel3}, giving $\tilde{\bu}_h\in \widetilde{\bU}_h$. 
It can be clearly seen that optimal order of convergence
is reached for all fields in both polynomial degrees $k=1$ and $k=2$, 
which confirms the sharpness of the theoretical error bounds.

\medskip\noindent\textbf{Example 2.} 
Secondly, we test the properties of the {\it a posteriori} error estimator 
\eqref{globalestimator2}, including the reliability, efficiency, as well as quasi-efficiency of the estimator. 
In a first instance (Example 2A) we simply compute locally the estimator and check, using smooth exact solutions 
in a convex domain $\Omega = (0,1)^2$, how it relates to the true error, by refining uniformly the mesh. 
Defining the smooth function $\varphi(x,y):= x^2(1-x)^2y^2(1-y)^2$, 
 the closed-form solutions are 
\begin{gather*}
\bu(x,y):= \curl \varphi, \quad p(x,y):= x^4-y^4, \quad   \bomega(x,y):= \sqrt{\nu} \curl\bu,\end{gather*}
and we take $\nu = 10^{-3}$, $\sigma = 10$, and  $\bbbeta(x,y):=  \curl \varphi$. 
Only Dirichlet velocity conditions are considered in this example (that is, $\Gamma_2$ is empty), which 
amounts to add a real Lagrange multiplier imposing the condition of zero-average for the Bernoulli pressure.  In Table~\ref{table:aposte} we collect 
the error history of the method, including individual errors and convergence rates as well as the errors 
analysed in Theorems~\ref{th:reliability}, \ref{thm:quasi-efficiency}, \ref{thm:efficiency}. As the estimator and the quasi-efficiency depend on the values of $\delta$, we explore three cases $\delta \in\{1/10,1/2,1\}$. The robustness is assessed by computing the effectivity 
indexes as the ratios 
\begin{equation}\label{eq:indexes}
\texttt{eff}_1 : = \frac{\Vert(\sigma^{1/2}e_{\bomega},e_p) \Vert_{0,\O}}{\tilde\betta}, \qquad \texttt{eff}_2 : = \frac{\Vert h^{\delta}_{\cT_h}(e_{\bomega},e_p)\Vert_{\mV}}{\tilde\betta}.
\end{equation}
The results confirm that the estimator is robust with respect to the weighted $\mV$-norm for all values of $\delta$, but the second-last column of the table indicates that $\tilde\betta$ is not necessarily efficient in the $\L^2$-norm, for $\delta <1$. 

\begin{table}[!t]
\setlength{\tabcolsep}{2pt}
\begin{center}
{\footnotesize\begin{tabular}{|c|c|c|c|c|c|c|c|c|c|c|c|c|}
  \hline
DoF & $\Vert e_{\bomega}\Vert_{0,\O}$  & \texttt{rate}  & 
$\Vert e_p\Vert_{0,\O}$  & \texttt{rate}  & 
$\Vert\bu-\tilde{\bu}_h\Vert_{0,\O}$ & \texttt{rate}  &  $\Vert(\sqrt{\sigma}e_{\bomega},e_p) \Vert_{0,\O}$ & \texttt{rate} &  $\Vert h^{\delta}_{\cT_h}\!\!(e_{\bomega},e_p)\Vert_{\mV}$ & \texttt{rate} & \texttt{eff}$_1(\tilde\betta)$  & \texttt{eff}$_2(\tilde\betta)$  \\ \hline \hline
23  & 8.87e-05 & -- & 1.19e-04 & -- & 2.36e-04 &  -- & 0.00031 & -- &  0.00100 &  -- &  0.0065 & 0.0215 \\
53  & 0.000196 & -1.90 & 6.92e-05 & 1.31 & 2.04e-04 & 0.35 & 0.00062 & -1.72 &  0.00061 &  1.22 &  0.0344 & 0.0332 \\
101  & 9.79e-05 & 2.16 & 3.36e-05 & 2.24 & 1.86e-04 & 0.28 & 0.00031 & 2.16 &  0.00024 &  2.74 &  0.0368 & 0.0295 \\
151  & 6.11e-05 & 2.34 & 2.16e-05 & 2.18 & 1.10e-04 & 2.61 & 0.00019 & 2.34 &  0.00015 &  2.33 &  0.0672 & 0.0539 \\
333  & 1.56e-05 & 3.46 & 8.83e-06 & 2.27 & 3.59e-05 & 2.84 & 5.01e-05 & 3.43 &  7.07e-05 &    2.02 &  0.0463 & 0.0654 \\
625  & 6.55e-06 & 2.75 & 6.25e-06 & 1.10 & 1.64e-05 & 2.48 & 2.17e-05 & 2.66 &  5.08e-05 &  1.05 &  0.0414 & 0.0973 \\
1493  & 3.08e-06 & 1.73 & 2.38e-06 & 2.21 & 7.87e-06 & 1.69 & 1.02e-05 & 1.76 &  2.39e-05 &  1.74 &  0.0408 & 0.0971 \\
2837  & 1.55e-06 & 2.15 & 1.19e-06 & 2.14 & 4.01e-06 & 2.11 & 5.04e-06 & 2.15 &  1.32e-05 &  1.84 &  0.0362 & 0.0989 \\
6285  & 6.99e-07 &   2.00 & 5.22e-07 & 2.09 & 1.83e-06 & 1.96 & 2.27e-06 & 2.01 &  6.74e-06 &  1.69 &  0.0345 & 0.1020 \\
14631  & 3.42e-07 & 1.69 & 2.29e-07 & 2.03 & 8.24e-07 & 1.90 & 1.11e-06 & 1.71 &  3.24e-06 &  1.73 &  0.0340 & 0.1020 \\
28095  & 1.81e-07 & 1.95 & 1.12e-07 & 2.09 & 4.50e-07 & 1.85 & 5.83e-07 & 1.95 &  1.86e-06 &  1.72 &  0.0318 & 0.1020 \\
63113  & 8.46e-08 & 1.88 & 4.96e-08 & 2.02 & 2.02e-07 & 1.98 & 2.72e-07 & 1.88 &  9.23e-07 &  1.73 &  0.0295 & 0.1019 \\
\hline
\end{tabular}}
\end{center}
\caption{Example~2B. Error history and effectivity indexes \eqref{eq:indexes} associated with the {\it a posteriori} error estimator \eqref{globalestimator2} using 
$\delta  =2/3$. Steep solutions on an L-shaped domain. 
Approximation with $k = 1$, and velocity postprocessing using \eqref{eq:vel3}.}  \label{table:aposte-L}
\end{table}

\begin{figure}[t]
\begin{center}
\includegraphics[width=0.245\textwidth]{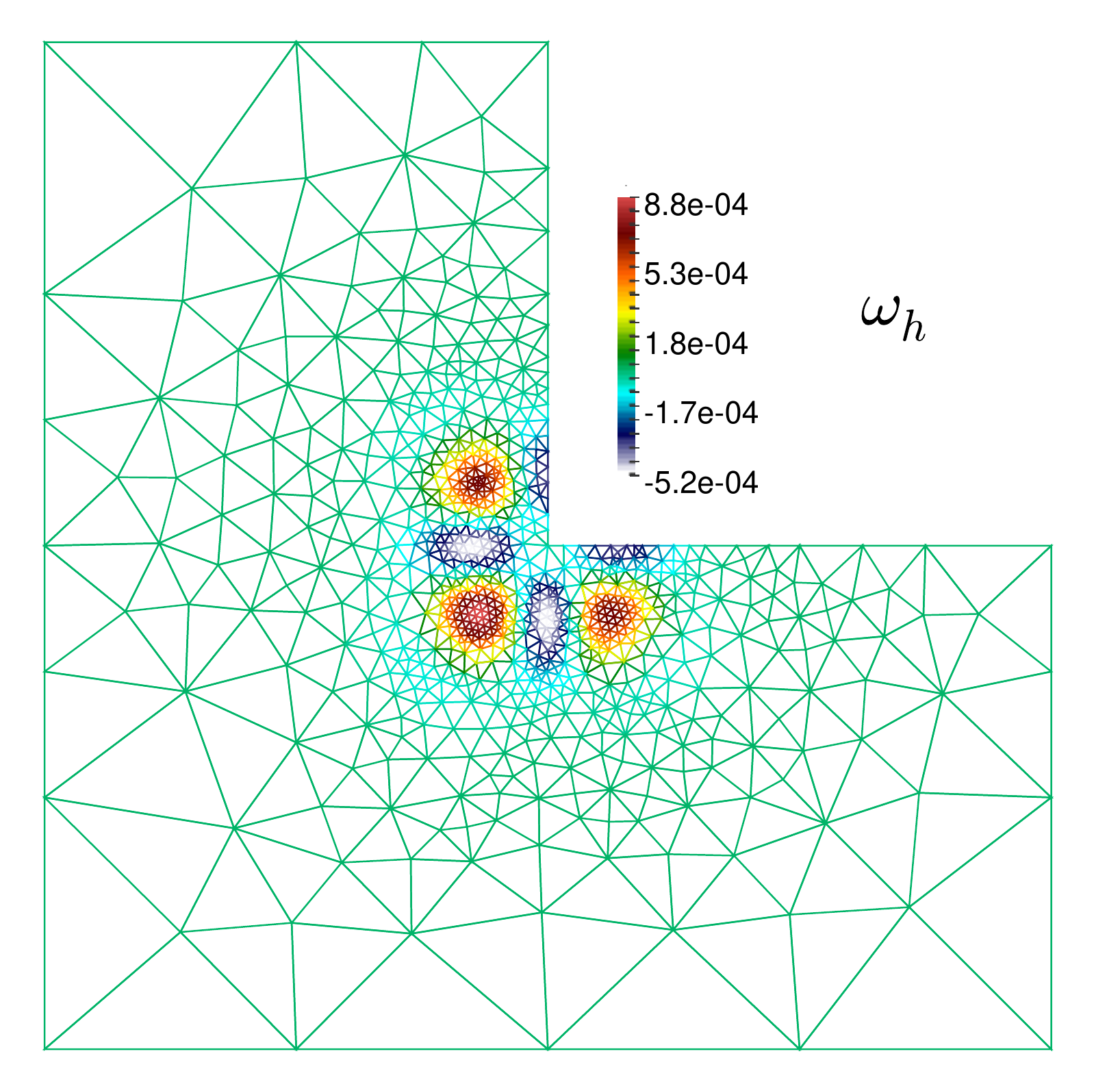}
\includegraphics[width=0.245\textwidth]{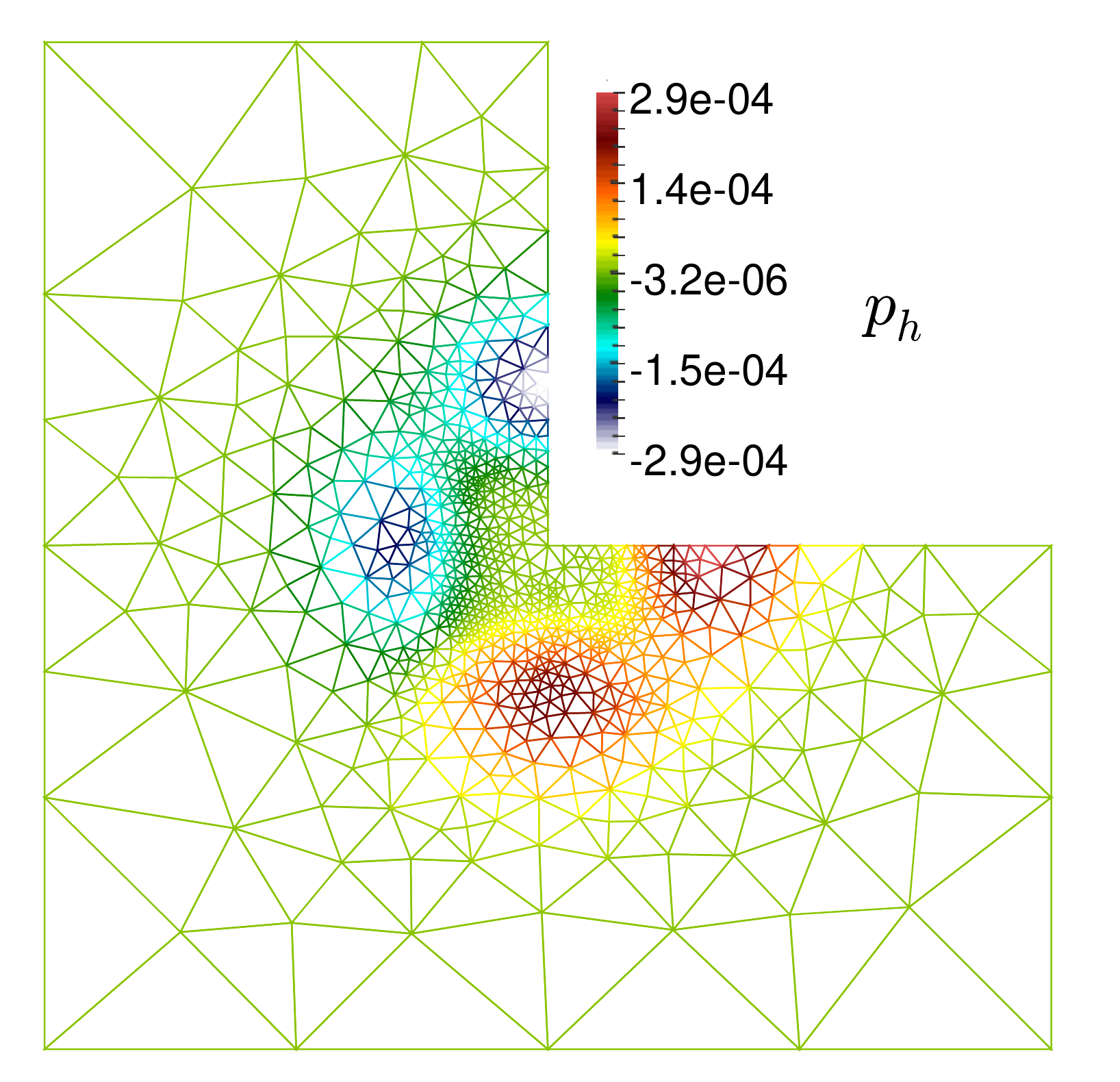}
\includegraphics[width=0.245\textwidth]{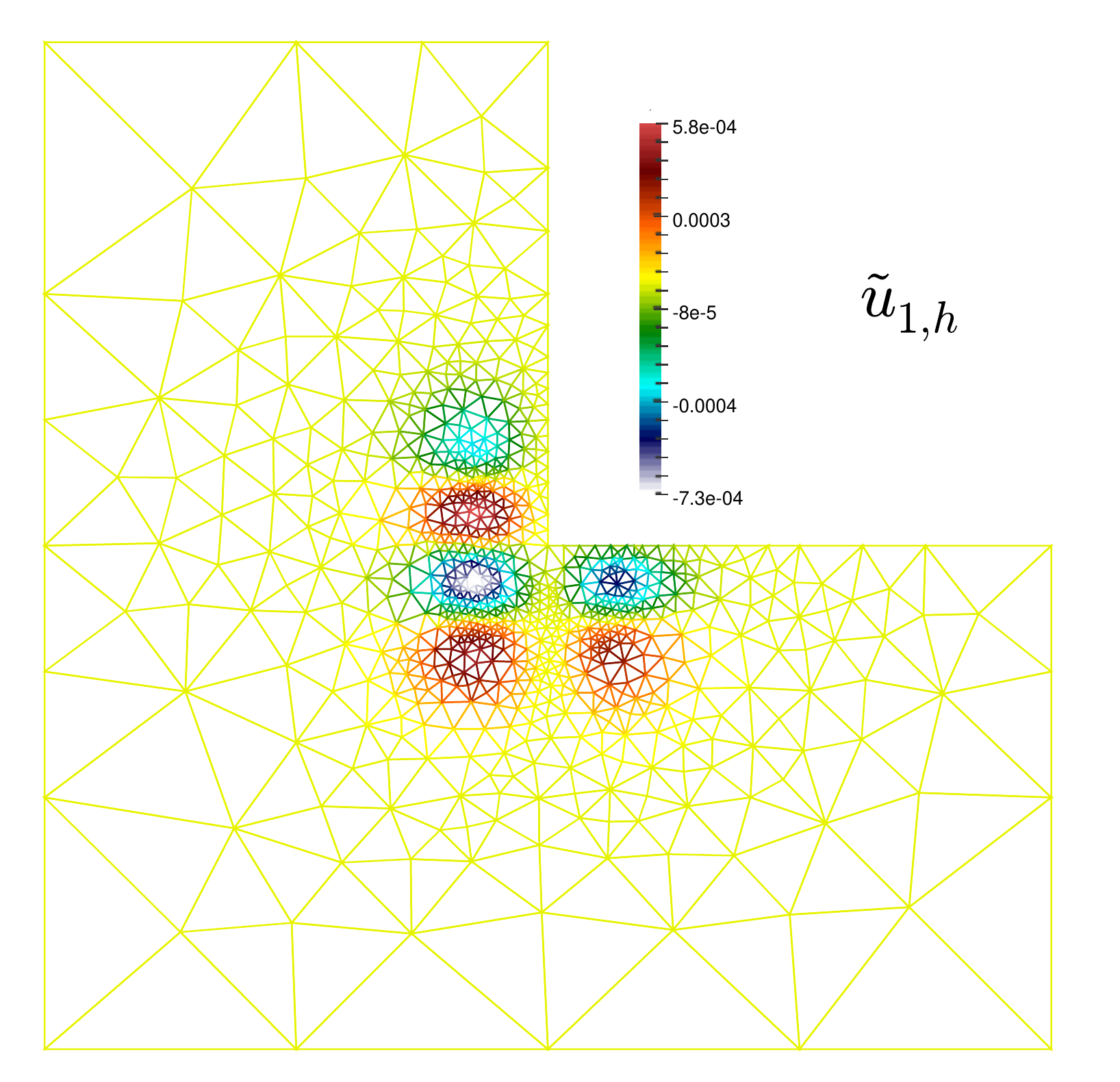}
\includegraphics[width=0.245\textwidth]{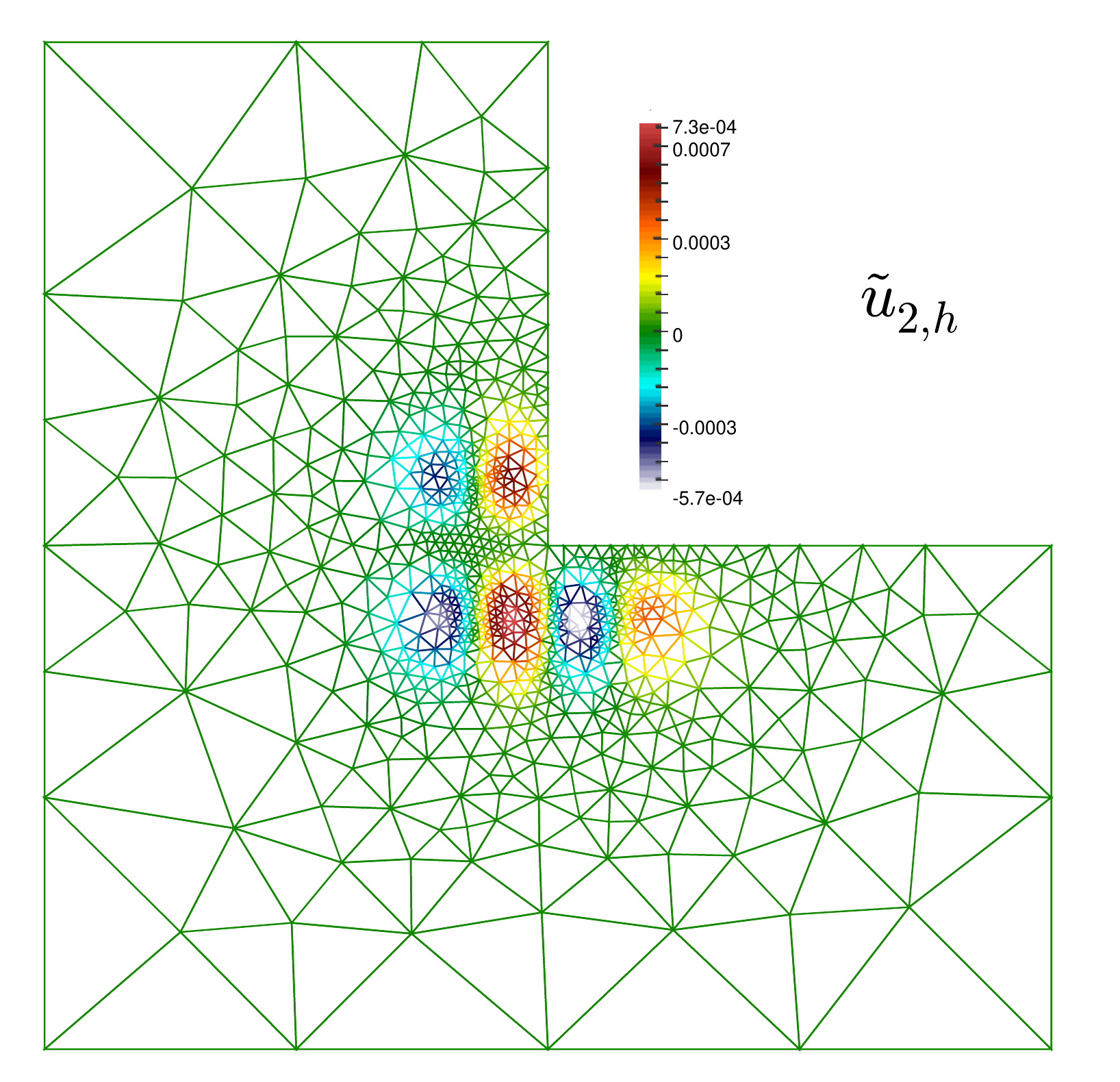}
\end{center}
\caption{Example~2B. Approximate vorticity, Bernoulli pressure, and velocity components 
obtained from \eqref{eq:vel3}. Solutions computed after six steps of adaptive mesh refinement following \eqref{globalestimator2} 
with $\delta = 2/3$.}  \label{fig:ex02b-sols}
\end{figure}

Next, as Examples 2B and 2C, we consider exact solutions with higher gradients and see how the estimator performs guiding adaptive 
mesh refinement as well as restoring optimal convergence rates. 
For this we follow a standard procedure 
of solving the discrete problem $\rightarrow$ estimating the error $\rightarrow$ marking cells for refinement $\rightarrow$ refining the mesh $\rightarrow$ solving again.  The marking is based on the equi-distribution 
of the error in such a way that the diameter of each new element   
(contained in a generic triangle $T$ on the initial coarse mesh) is proportional to the 
initial diameter times the ratio $\bar{\tilde\betta}_h / \boldsymbol{\eta}_{T}$, 
where $\bar{\tilde\betta}_h$ is  the mean value of $\tilde\betta$  over the initial 
mesh \cite{rv-1996}. The refinement is then done on the marked elements as well as 
on an additional small layer in order to maintain the regularity of the resulting grid. An extra smoothing 
step is also applied after the refinement step. 

For Example 2B we concentrate on the L-shaped domain $\Omega = (-1,1)^2\setminus(0,1)^2$, and use the exact solutions
\begin{gather*}
\varphi(x,y):= x^2(1-x)^2y^2(1-y)^2 \exp(-50(x-0.01)^2-50(y-0.01)^2), \quad \bu(x,y):= \curl \varphi, \\ 
p(x,y):= (x^5-y^5)\exp(-25(x-0.01)^2-25(y-0.01)^2), \quad   \bomega(x,y):= \sqrt{\nu} \curl\bu,\end{gather*}
employed also to compute boundary data and right-hand side forcing terms. We keep the values of $\nu,\sigma$ from Example 1. The regularity of the coupled problem (due to the corner singularity) indicates that $\delta= 2/3$.
We collect the results in Table~\ref{table:aposte-L}, showing similar trends as those seen in Table~\ref{table:aposte}, that is, optimal convergence for all fields, and robustness of the {\it a posteriori} error estimator in the $\mV$-norm. Samples of 
approximate vorticity, Bernoulli pressure, and post-processed velocity, also for the case of $\delta = 2/3$, and after six steps of adaptive mesh refinement are shown in Figure~\ref{fig:ex02b-sols}. 

\begin{figure}[t]
\begin{center}
\includegraphics[width=0.45\textwidth]{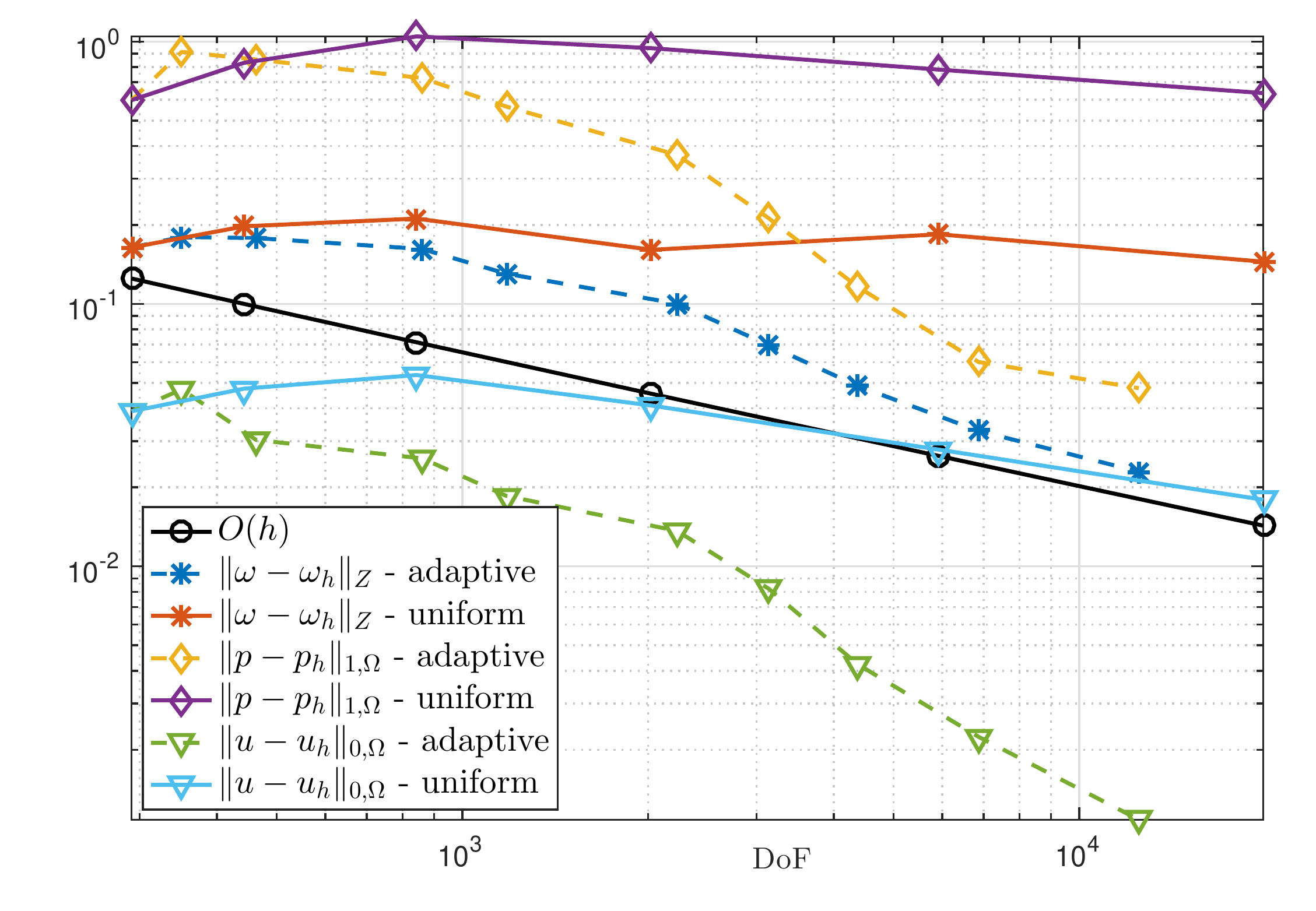}
\includegraphics[width=0.45\textwidth]{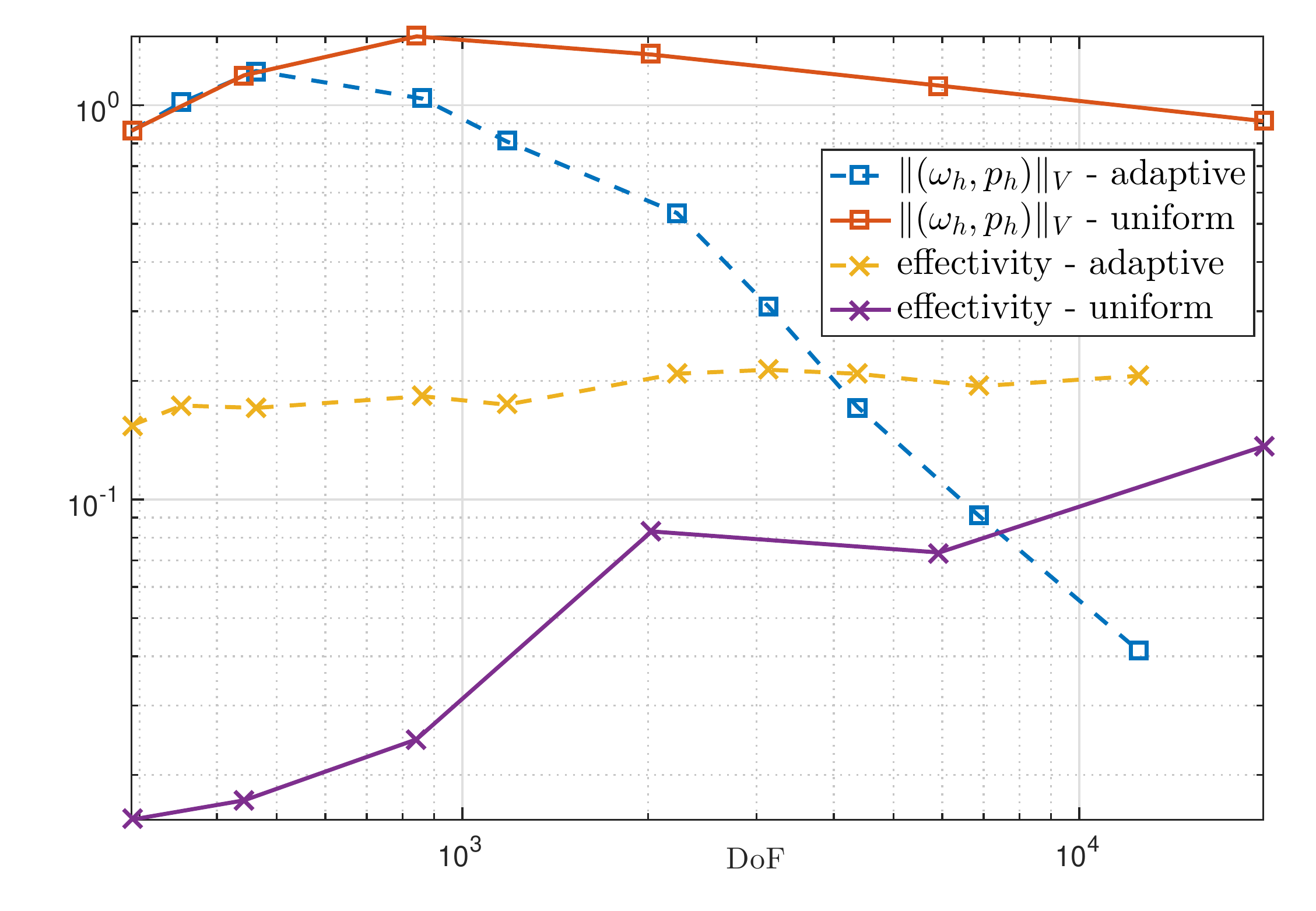}\\[1ex]
\includegraphics[width=0.3\textwidth]{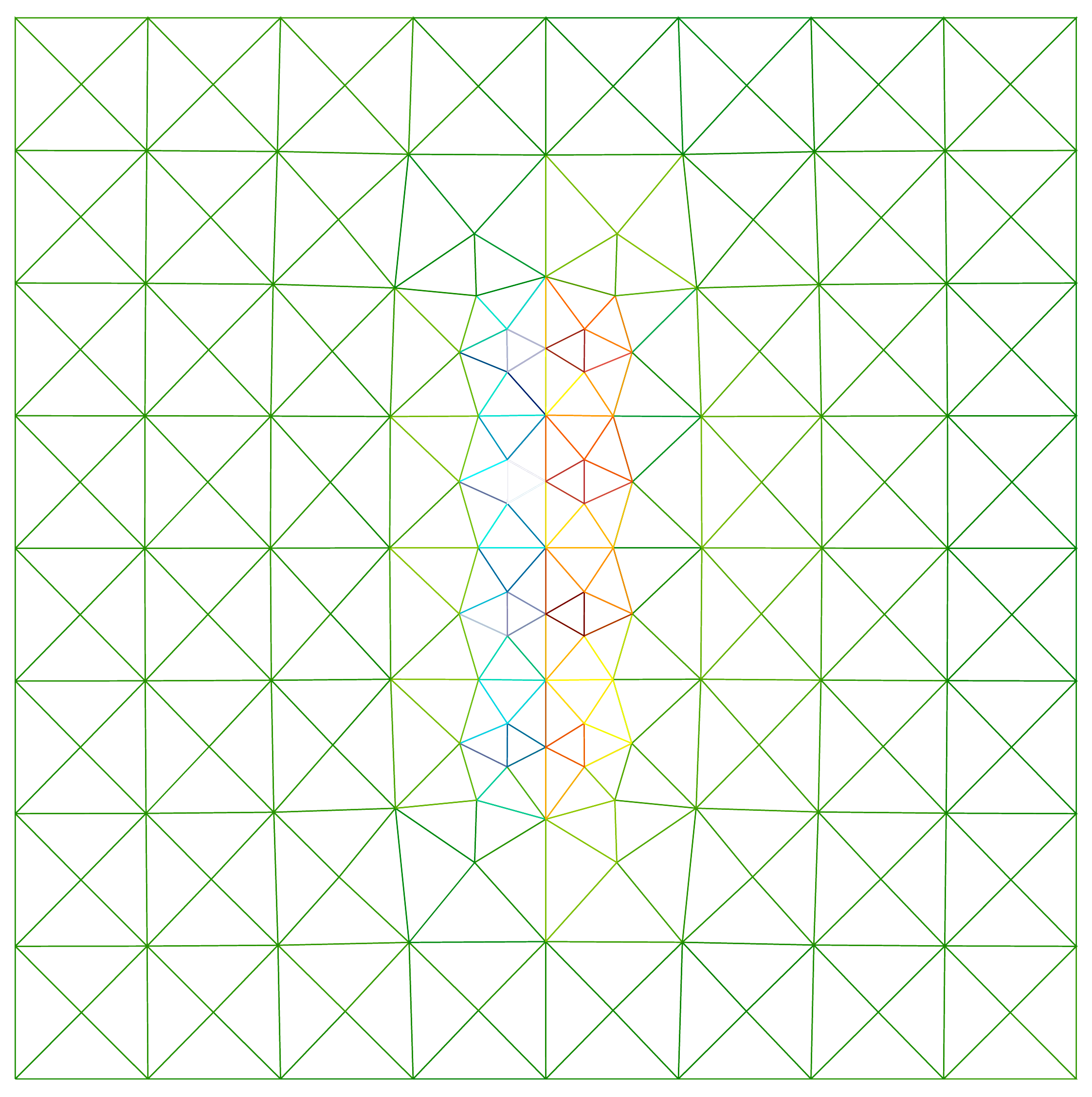}
\includegraphics[width=0.3\textwidth]{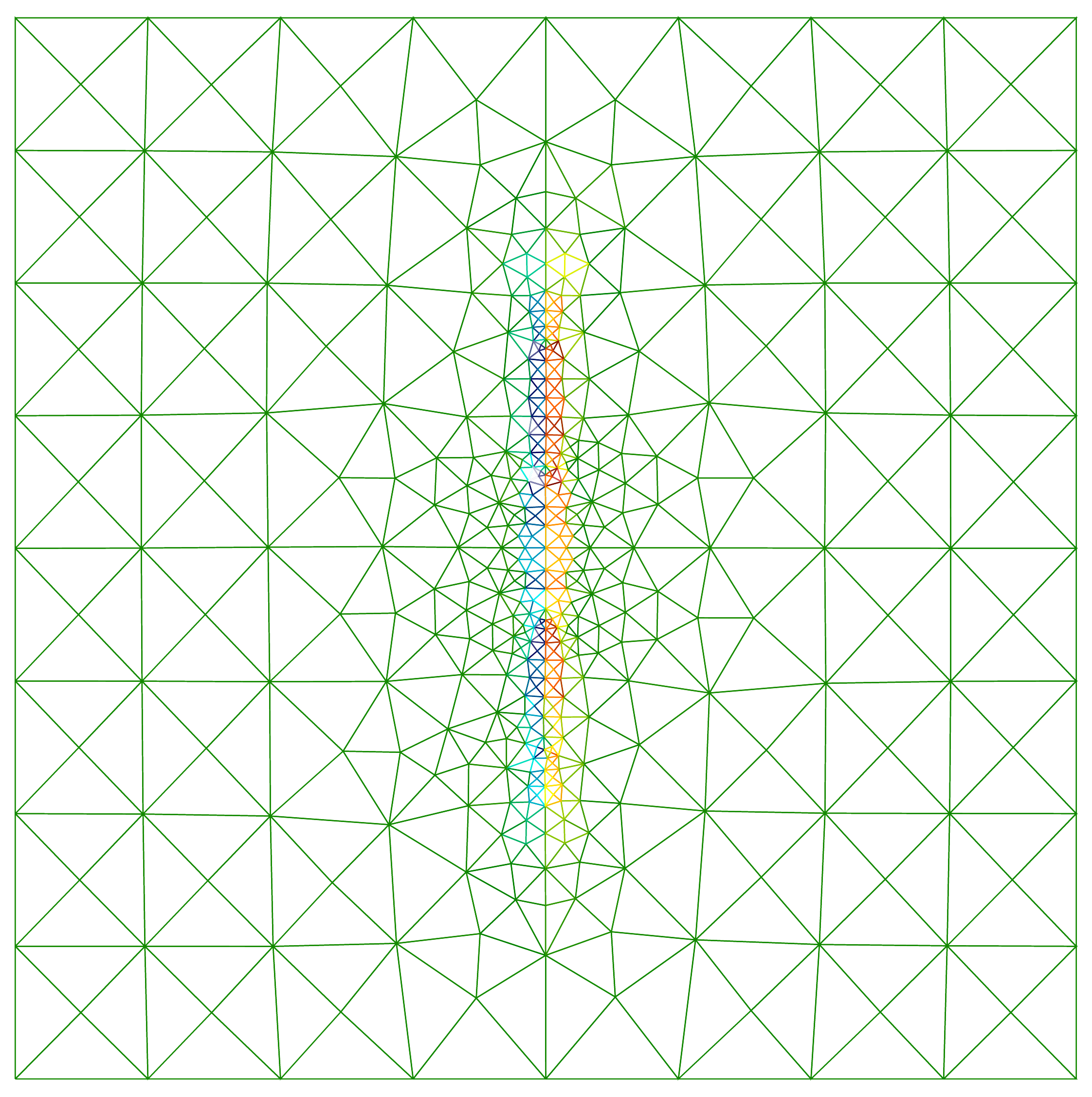}
\includegraphics[width=0.3\textwidth]{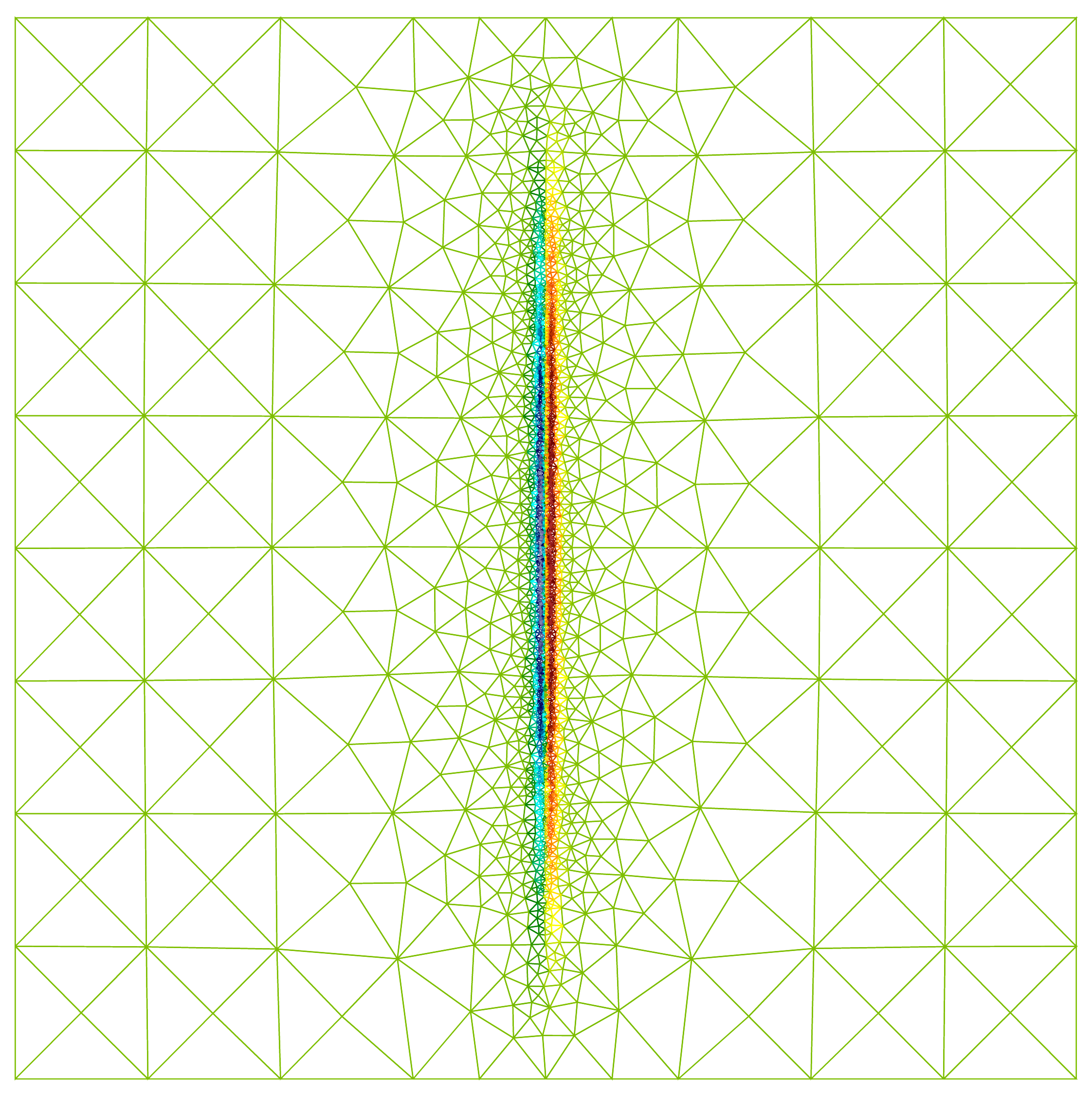}\\
\includegraphics[width=0.25\textwidth]{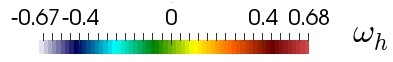}
\end{center}
\caption{Example~2C. Error decay in different norms and effectivity index $\texttt{eff}_2$ for the finite element approximation of the Oseen equations 
having an inner layer. Comparison plots between uniform (solid lines) and adaptive (dashed lines)  
mesh refinement using the lowest-order scheme (top panels); and examples of 
meshes produced after one, three, and six steps of adaptive refinement (bottom row).}  \label{fig:ex2}
\end{figure}

For Example 2C, starting from a coarse initial triangulation of the domain, we construct sequences 
of uniformly and adaptively refined meshes and compute errors between approximate solutions 
and the following closed-form solutions exhibiting a vertical inner layer near the central axis of the domain 
(see \cite{barrios}) 
\begin{gather*}
\varphi(x,y):= x^2(1-x)^2y^2(1-y)^2[1-\tanh(150(1/2-x))], \quad \bu(x,y):= \curl \varphi, \\ 
p(x,y):= e^{-(x-1/2)^2}-p_0, \quad   \bomega(x,y):= \sqrt{\nu} \curl\bu,\end{gather*}
 where $p_0$ is such the average of $p$ over $\Omega$ is zero, and we take $\nu = 10^{-4}$, $\sigma = 10$, and  $\bbbeta(x,y):=  \curl \varphi$. Again we take Dirichlet velocity conditions everywhere on $\partial\Omega$. 
 
Figure~\ref{fig:ex2} shows the error history in both cases, confirming that the method constructed 
upon adaptive mesh refinement provides rates of convergence slightly better than the theoretical optimal, 
whereas under uniform refinement the lack of smoothness in the exact solutions  hinder substantially the error decay, 
exhibiting sublinear convergence in all cases and even stagnating for vorticity. 
The top left plot portrays the individual errors, and for reference the optimal error decay for the case of less regular solutions (that is, $O(h)$); whereas the 
right panel shows the error in the $\mV$-norm and the effectivity index $\texttt{eff}_2$ defined in \eqref{eq:indexes}.   
In addition, the bottom panels of Figure~\ref{fig:ex2} display the outputs 
of mesh refinement indicating a higher concentration of elements where the large gradients 
are located.


\begin{figure}[t]
\begin{center}
\includegraphics[width=0.4\textwidth]{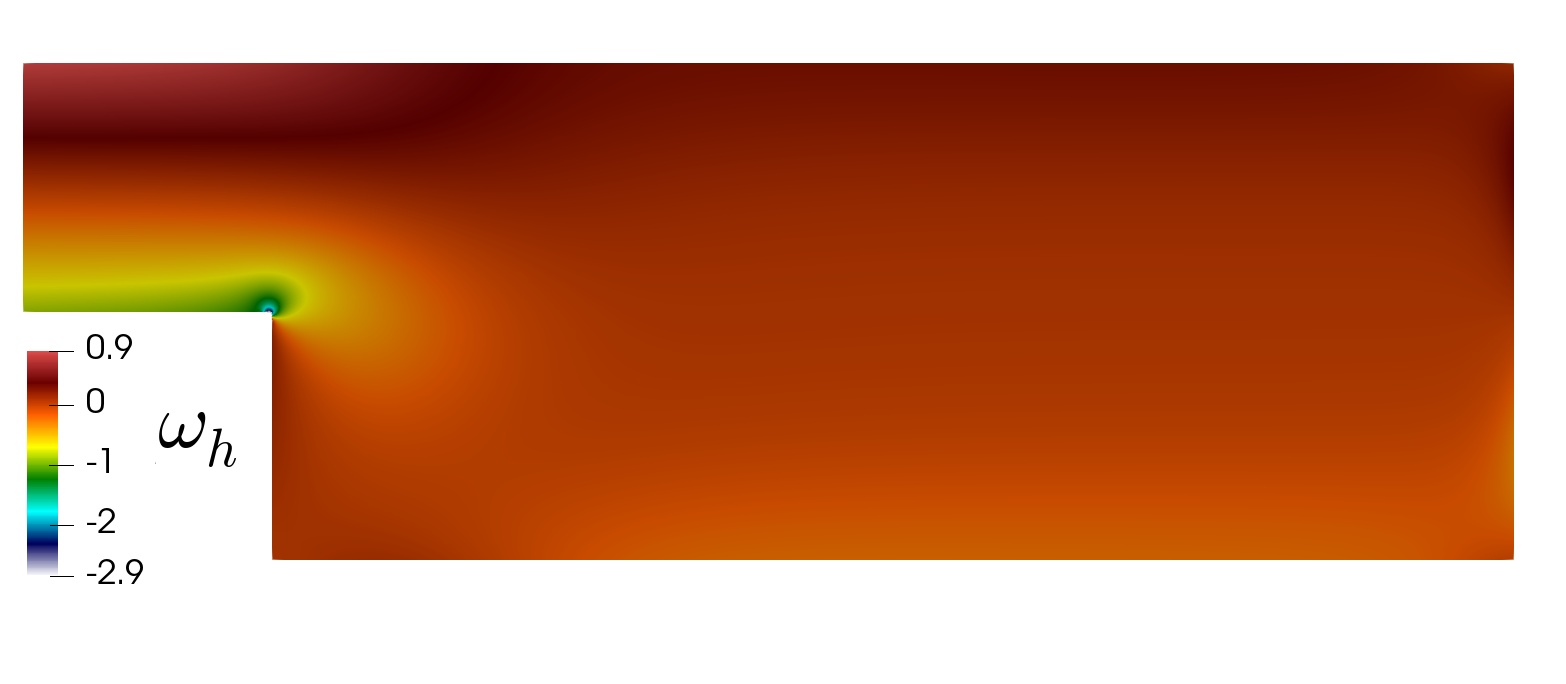}
\includegraphics[width=0.4\textwidth]{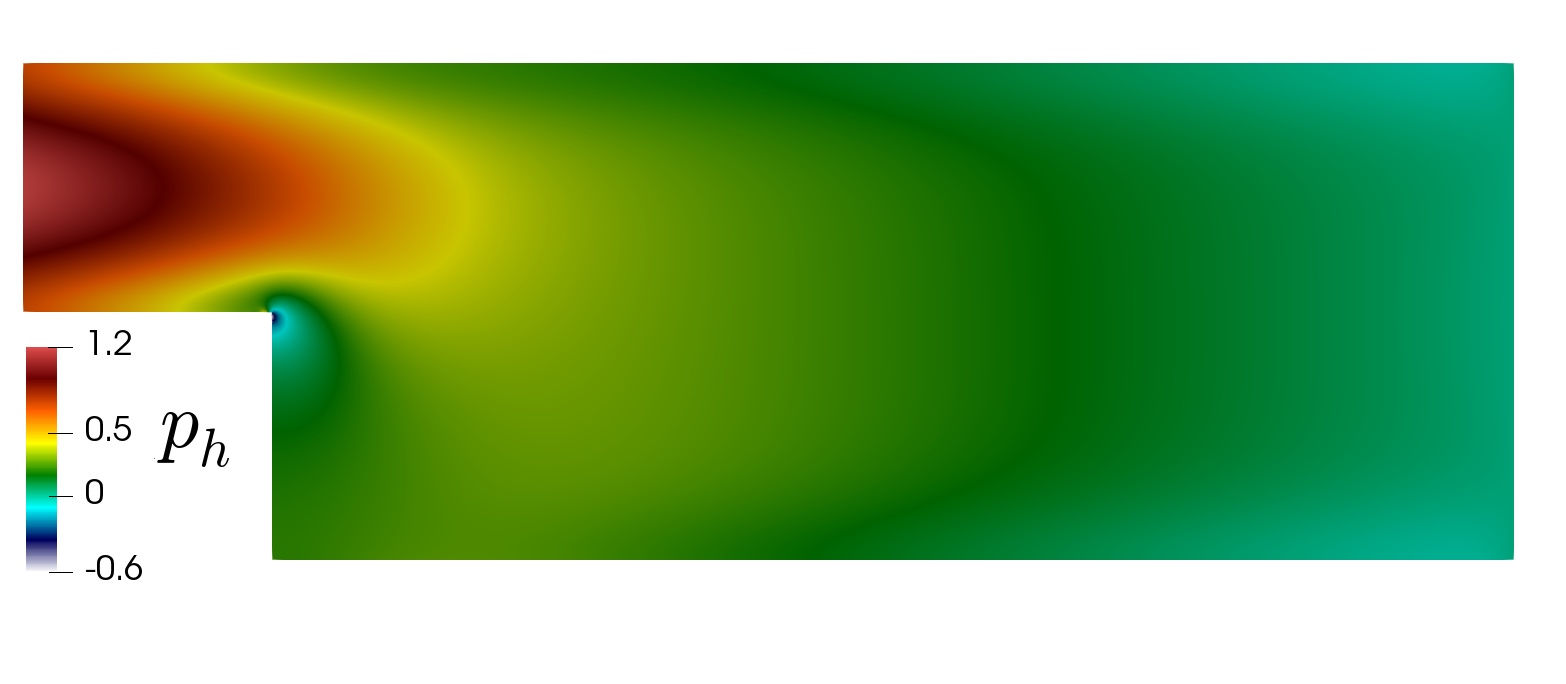}\\
\includegraphics[width=0.4\textwidth]{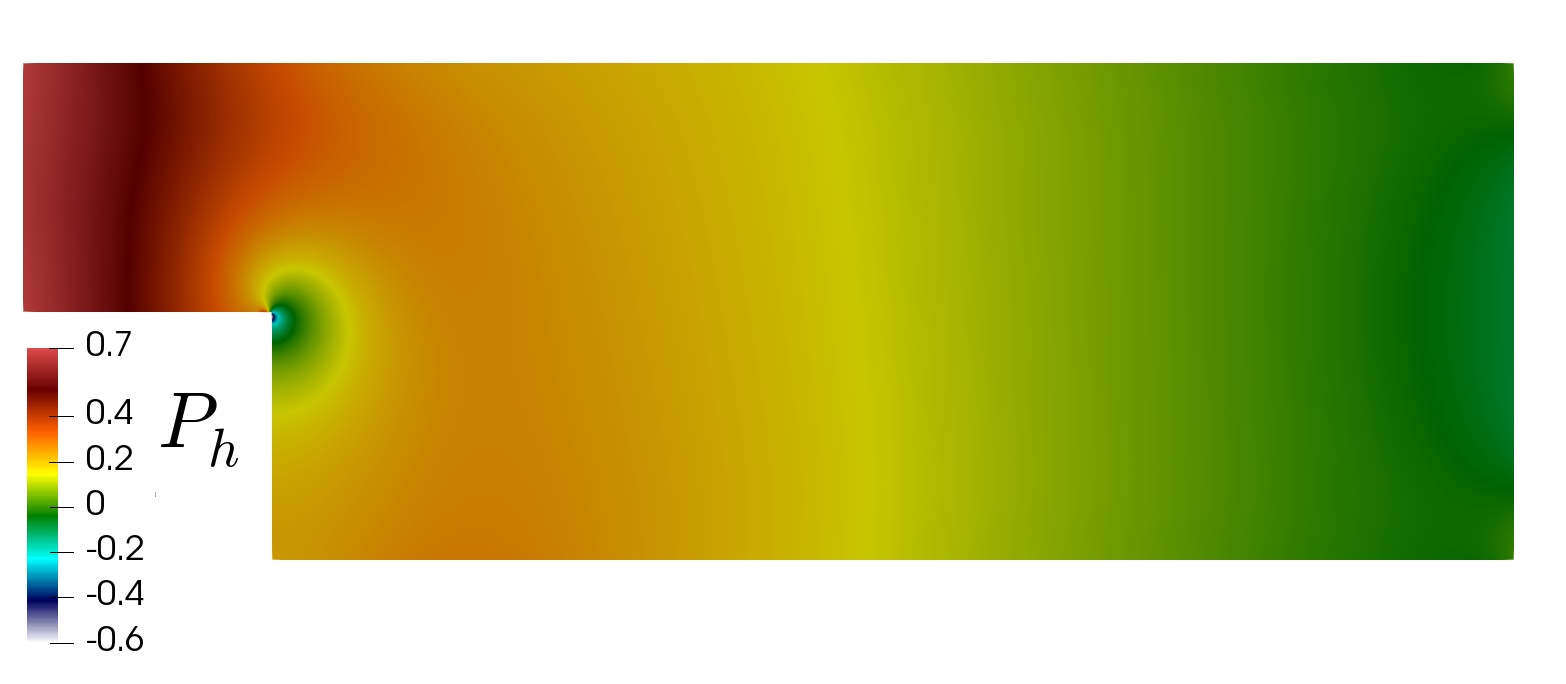}
\includegraphics[width=0.4\textwidth]{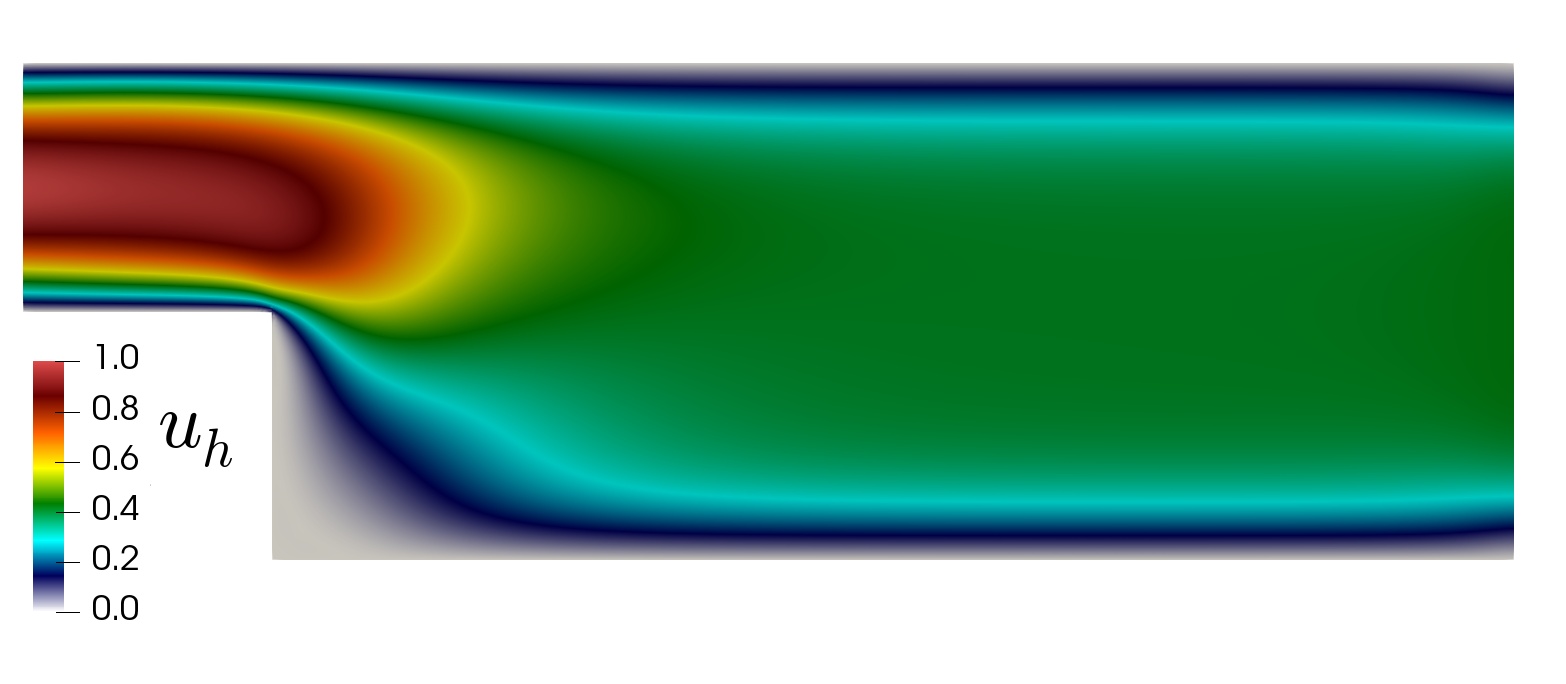}
\raisebox{5mm}{\includegraphics[width=0.18\textwidth]{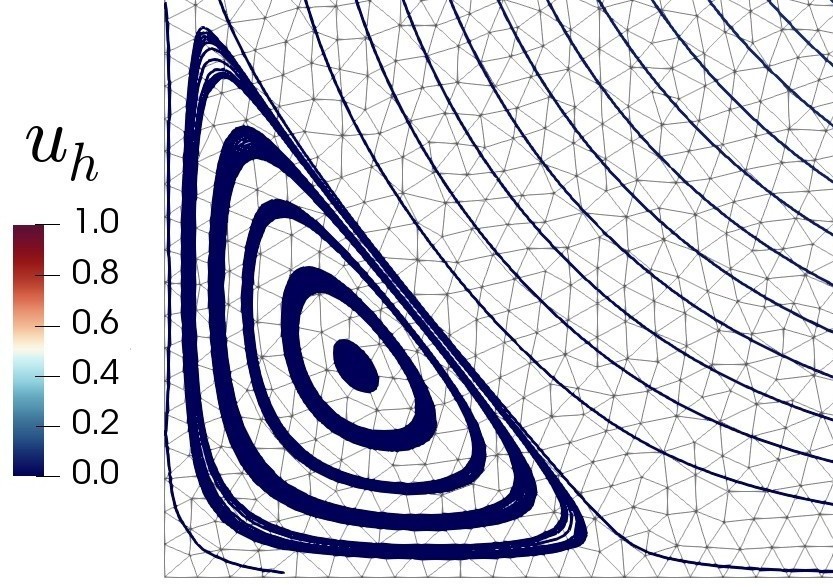}}
\end{center}
\caption{Example 3. Flow over a backward-facing step. Vorticity, Bernoulli pressure, true pressure, post-processed velocity, and zoom-in on bottom-left corner with velocity streamlines.}\label{ex02:fig}
\end{figure}

\medskip\noindent\textbf{Example 3.} 
Next, we conduct the well-known test of flow past a backward-facing step. This is also a 
2D example where the domain is $\Omega = (0,6)\times(0,2) \setminus (0,1)^2$. For this case 
we choose a method with $k=2$ and assume that $\bbbeta$ is the discrete velocity at the previous time 
iteration of a backward Euler time step. Assuming that no external forces are applied, we then have 
$\ff = \sigma\bbbeta$ and after each time step characterised by $\sigma = (\Delta t)^{-1} = 100$, we update 
the current velocity $\bbbeta \leftarrow \bu$.  The flow regime is determined by a moderate viscosity $\nu = 0.05$ and we prescribe 
$\Gamma_2$ as the right edge (the outlet of the channel) where we set $p_0 = 0$ and $\ba = \cero$. 
The remainder of the boundary constitutes $\Gamma_1$: on the left edge (the inlet of the channel) we  
impose a parabolic profile $\bg = (4(y-1)(2-y),0)^T$ and on the remainder of $\Gamma_1$ (the channel 
walls) we set $\bg = \cero$. The system is run until the final time $t = 1$ and samples of the 
obtained numerical results are collected in Figure~\ref{ex02:fig}. As expected for this test, a fully developed 
profile (seen in the plot of post-processed velocity) exits the outlet while an important recirculation occurs on the 
bottom-left corner, right after the expanding region. The vorticity has a very high gradient on the reentrant corner 
of the channel, but this is well-captured by the numerical scheme. We also show Bernoulli pressure and the 
classical pressure (which coincides with the expected pressure profiles for this example). In addition, in Figure~\ref{ex02a:fig} 
we portray examples of adaptively refined meshes using the indicator \eqref{globalestimator2}. One can observe 
local refinement near the reentrant corner and at later times, a clustering of elements near the horizontal walls in the 
channel. 

\begin{figure}[t]
\begin{center}
\includegraphics[width=0.45\textwidth]{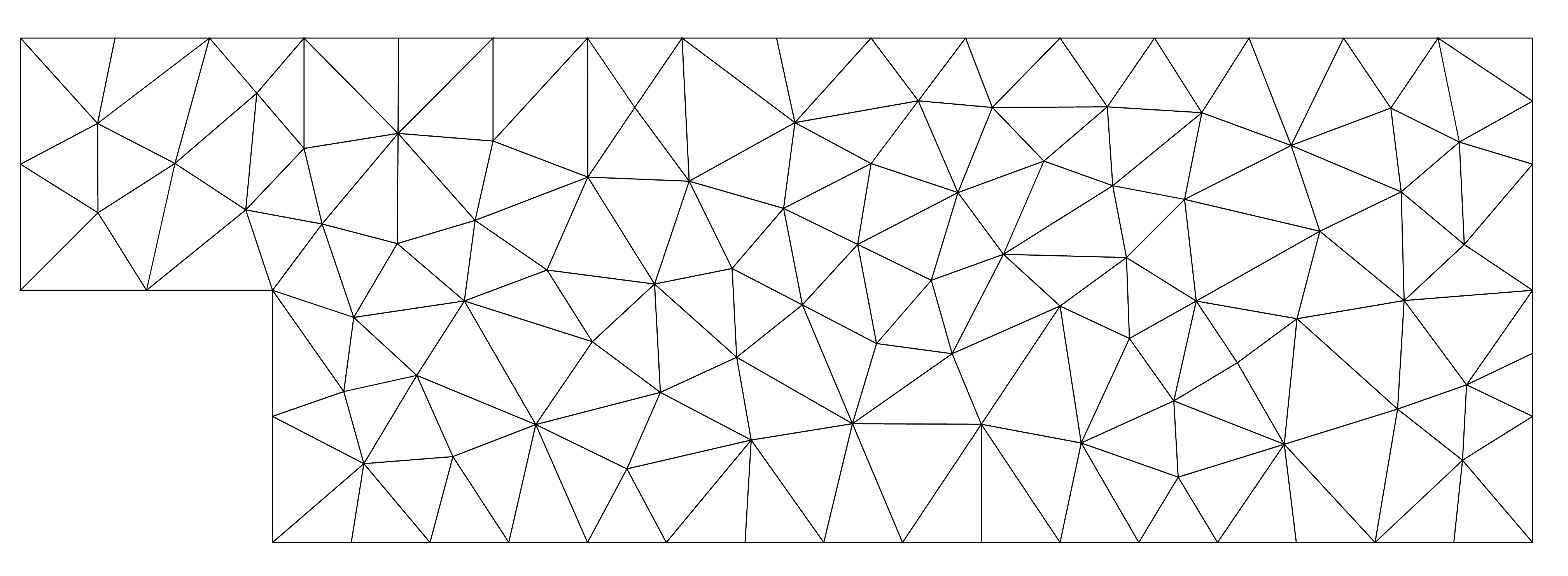}
\includegraphics[width=0.45\textwidth]{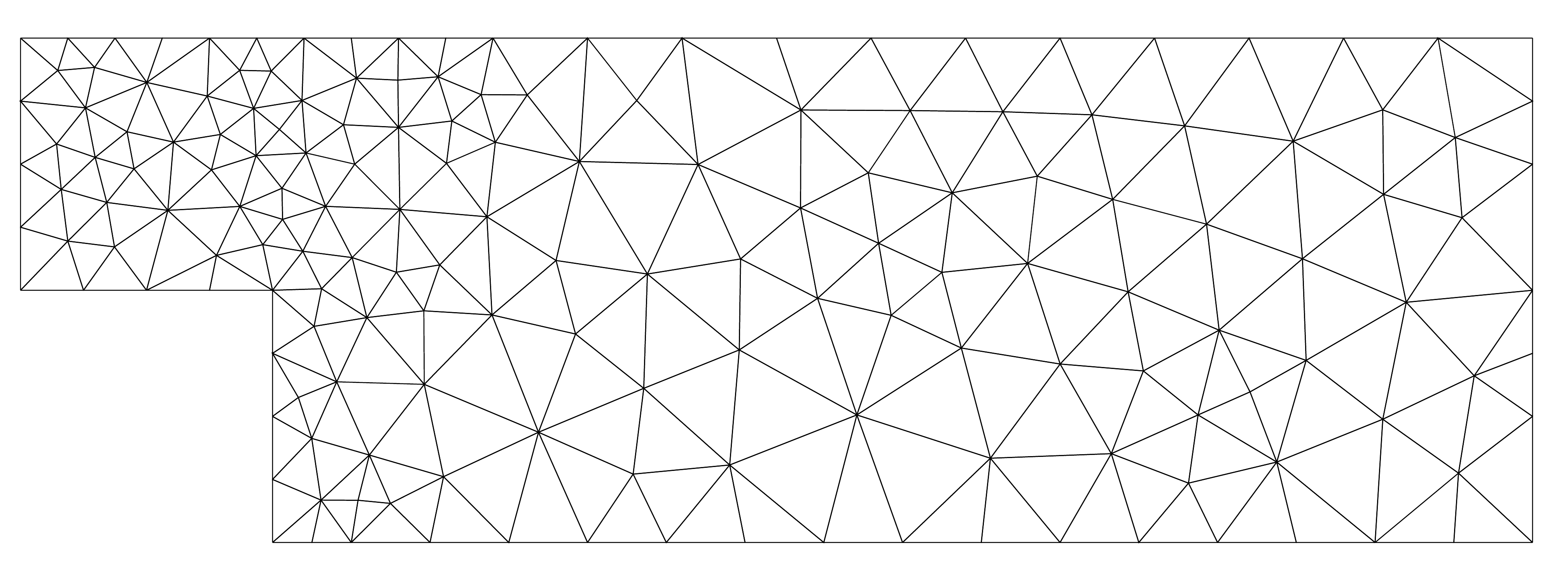}\\
\includegraphics[width=0.45\textwidth]{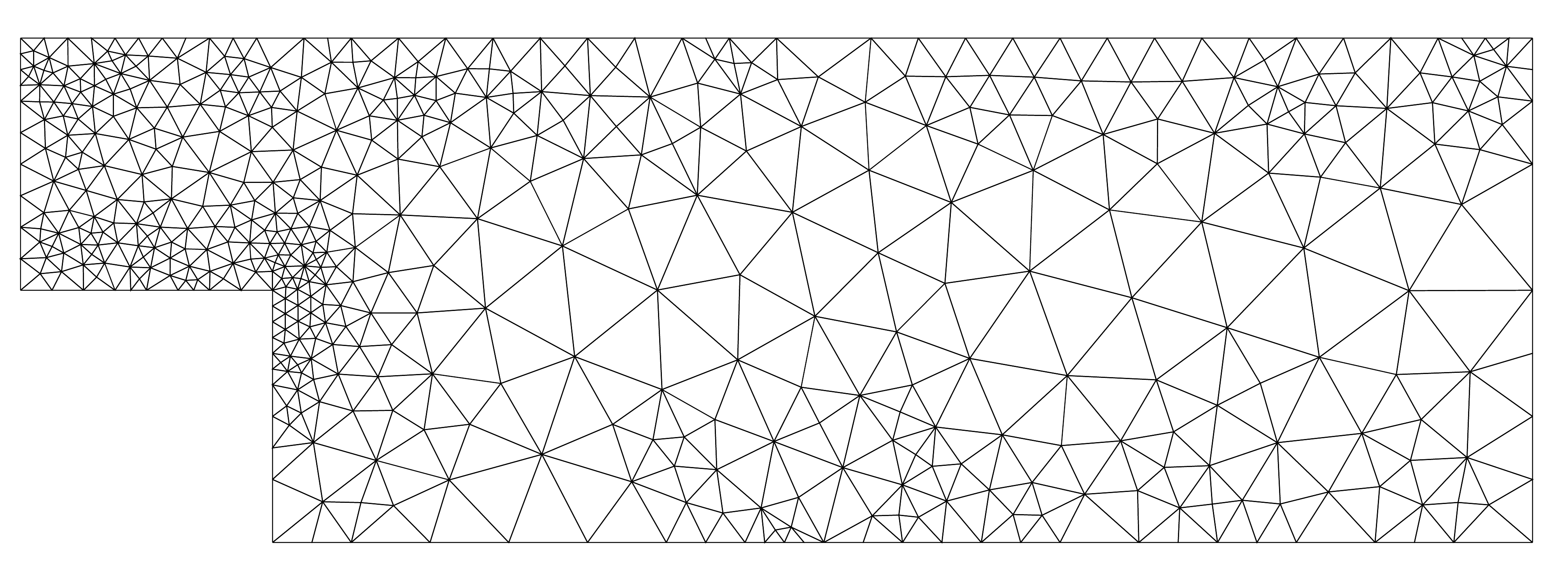}
\includegraphics[width=0.45\textwidth]{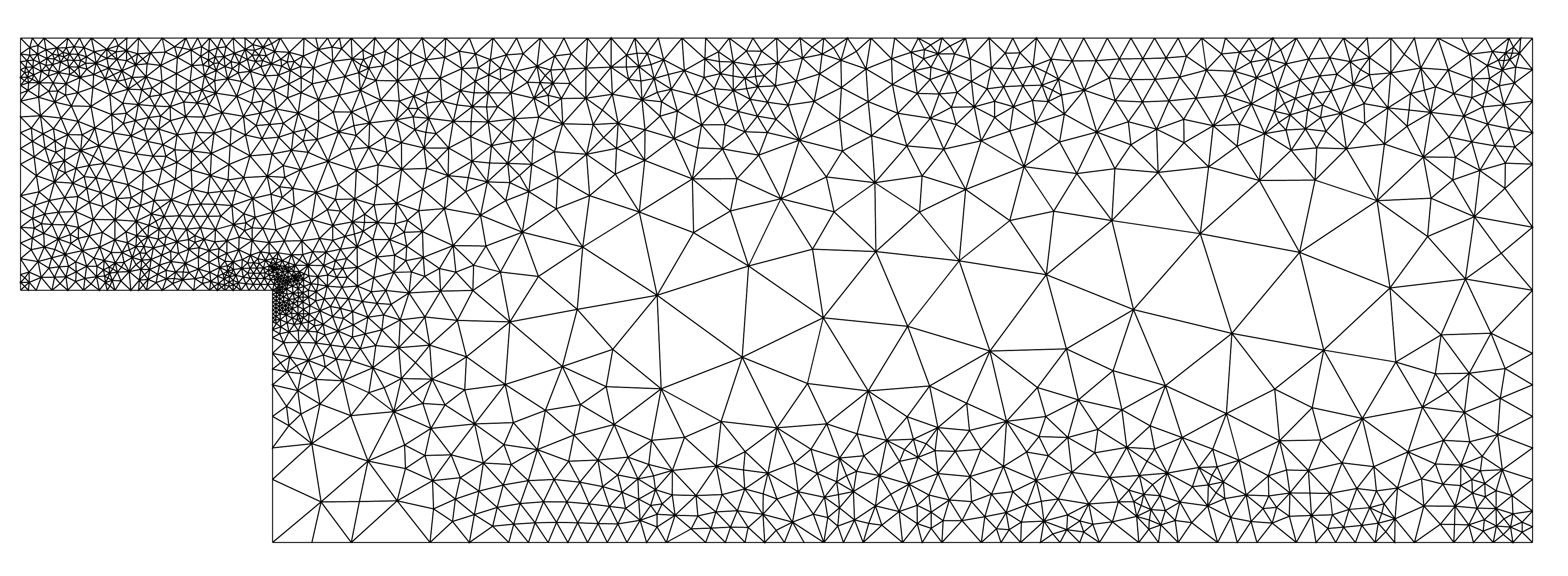}
\end{center}
\caption{Example 3. Flow over a backward-facing step. Adaptively refined meshes 
according to the {\it a posteriori} error indicator \eqref{globalestimator2}, applying up to four refinement steps (from top-left to right-bottom).}\label{ex02a:fig}
\end{figure}

\medskip\noindent\textbf{Example 4.} 
For our next application we study the flow patterns generated on a channel with three obstacles (using the domain and 
boundary configuration from the micro-macro models introduced in \cite{torr17}). Here the flow is now generated only through pressure 
difference between the inlet (the bottom horizontal section of the boundary defined by $(0,1)\times\{-2\}$) and the outlet (the vertical 
segment on the top left part of the boundary, defined by $\{-2\}\times(0,1)$). No other boundary conditions are set. As in the previous 
test case, $\bbbeta$ is the discrete velocity at the previous pseudo-time iteration. We take $\sigma = 10$ and 
$\nu = 0.02$ and increase the pressure at the inlet with the pseudo time, reaching after 10 steps the value 
$p_{\text{in}} = 3$ and set zero Bernoulli pressure at the outlet. The avoidance of the obstacles and accumulation 
of vorticity near them is a characteristic behaviour of the phenomenon that we can observe in Figure~\ref{ex03:fig}. 
These plots were generated with $k=2$. 

\begin{figure}[t]
\begin{center}
\includegraphics[width=0.325\textwidth]{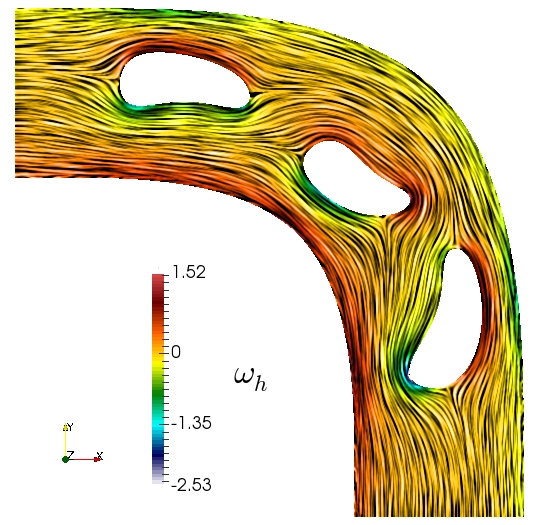}
\includegraphics[width=0.325\textwidth]{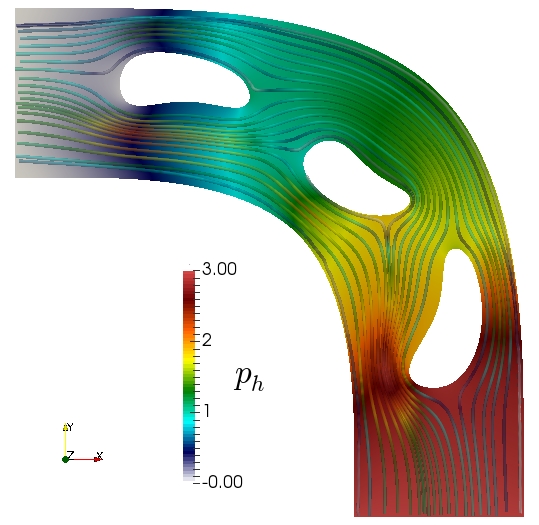}
\includegraphics[width=0.325\textwidth]{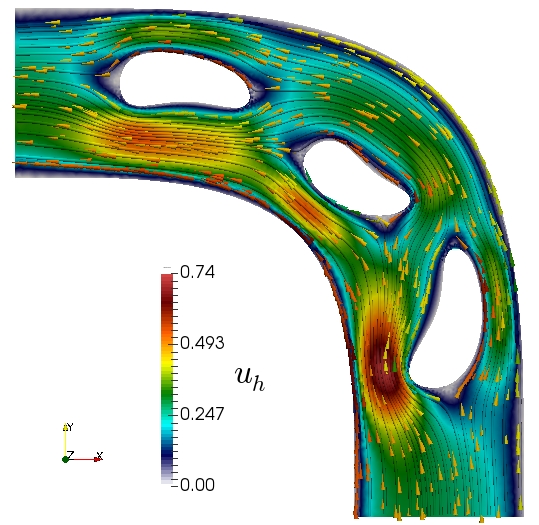}
\end{center}
\caption{Example 4. Flow inside a channel with obstacles. Vorticity and line integral contours, 
classical pressure together with velocity streamlines, and post-processed velocity magnitude and arrows. 
Computation done with a second-order method.}\label{ex03:fig}
\end{figure}

\medskip\noindent\textbf{Example 5.} 
Our last test exemplifies the performance of the numerical scheme 
in 3D. 
We use as computational domain the 
geometry of a femoral end-to-side bypass segmented from 3T MRI scans 
\cite{marchandise11}. We generate a volumetric mesh of 68351 
tetrahedra. The boundaries of this 
arterial bifurcation are considered as an inlet $\Gamma_{\text{in}}$, 
an outlet $\Gamma_{\text{out}}$, the arterial wall $\Gamma_{\text{wall}}$, and an 
occluded section $\Gamma_{\text{occl}}$. 
On the occlusion section and on the walls 
we set no-slip velocity. 
A parabolic velocity profile is considered at the inlet surface 
whereas a mean pressure distribution is prescribed 
on the outlet section. 
The last two conditions are time-dependent and periodic 
with a period of 50 time steps (we employ $\sigma = 100$ and 
run the system for 100 time steps). Moreover we use a blood viscosity of 
$\nu = 0.035$ (in g/cm$^3$), 
which represents an average Reynolds number between 144 and 
380 \cite{marchandise11}. The computations were carried out with the 
first-order scheme, and the results are shown in Figure~\ref{ex04:fig}, 
focusing on the solutions after 50 time steps. A relatively small zone 
with a secondary flow forms near the bifurcation, while the bulk stream 
continues towards the outlet.

\begin{figure}[t]
\begin{center}
\includegraphics[width=0.495\textwidth]{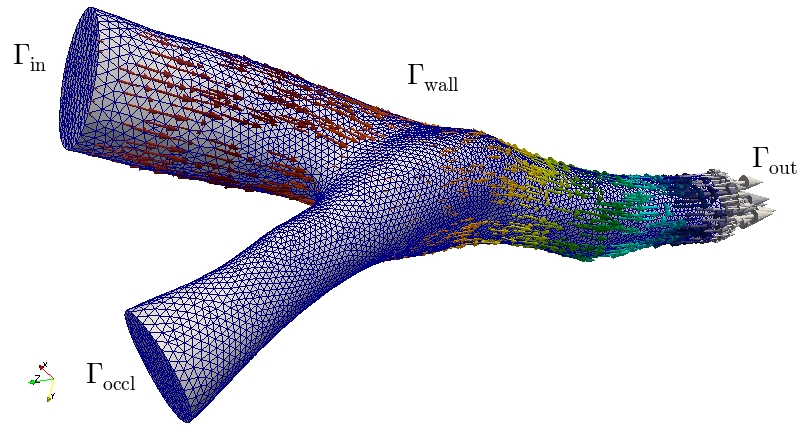}
\includegraphics[width=0.495\textwidth]{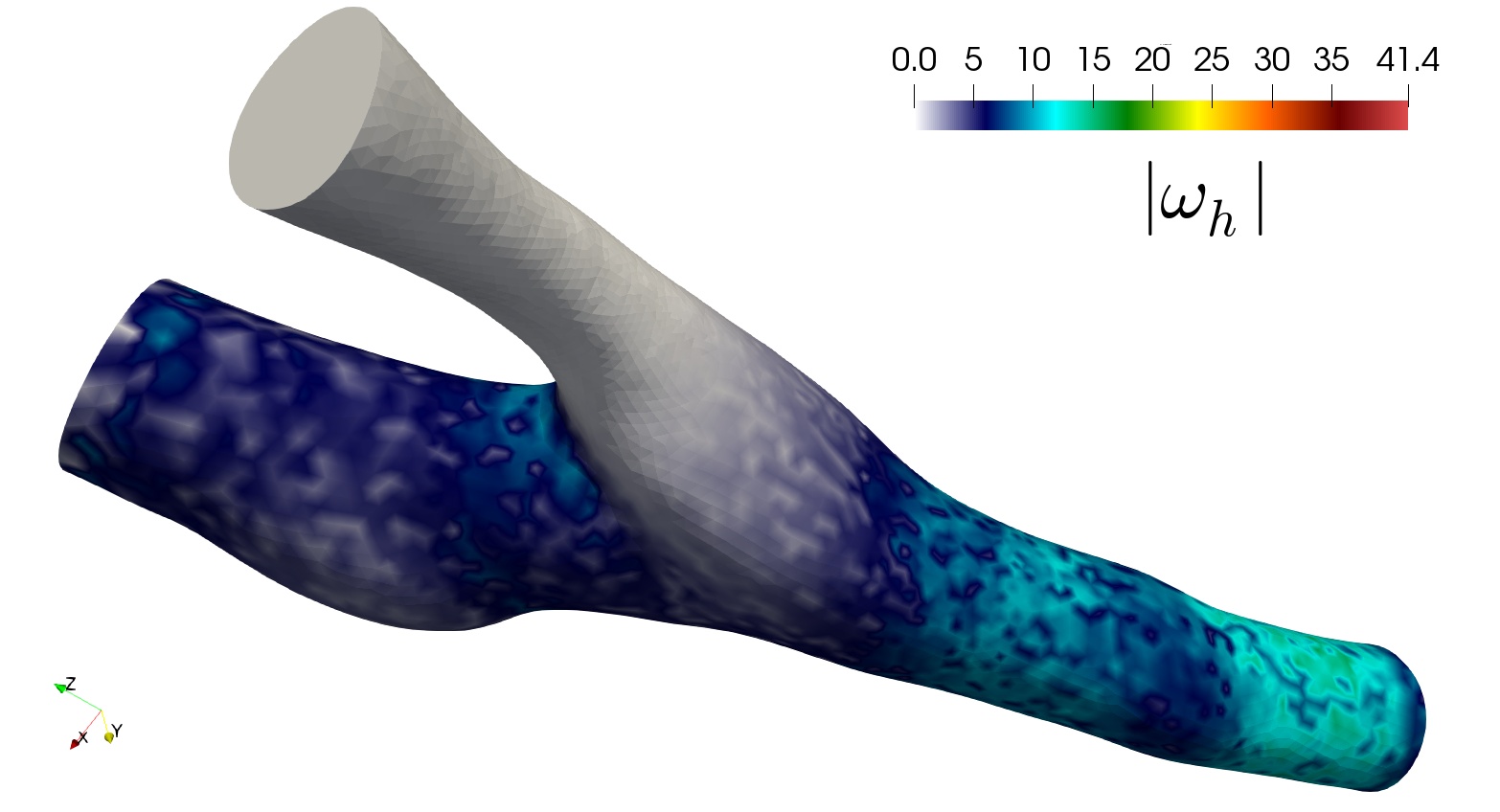}
\includegraphics[width=0.495\textwidth]{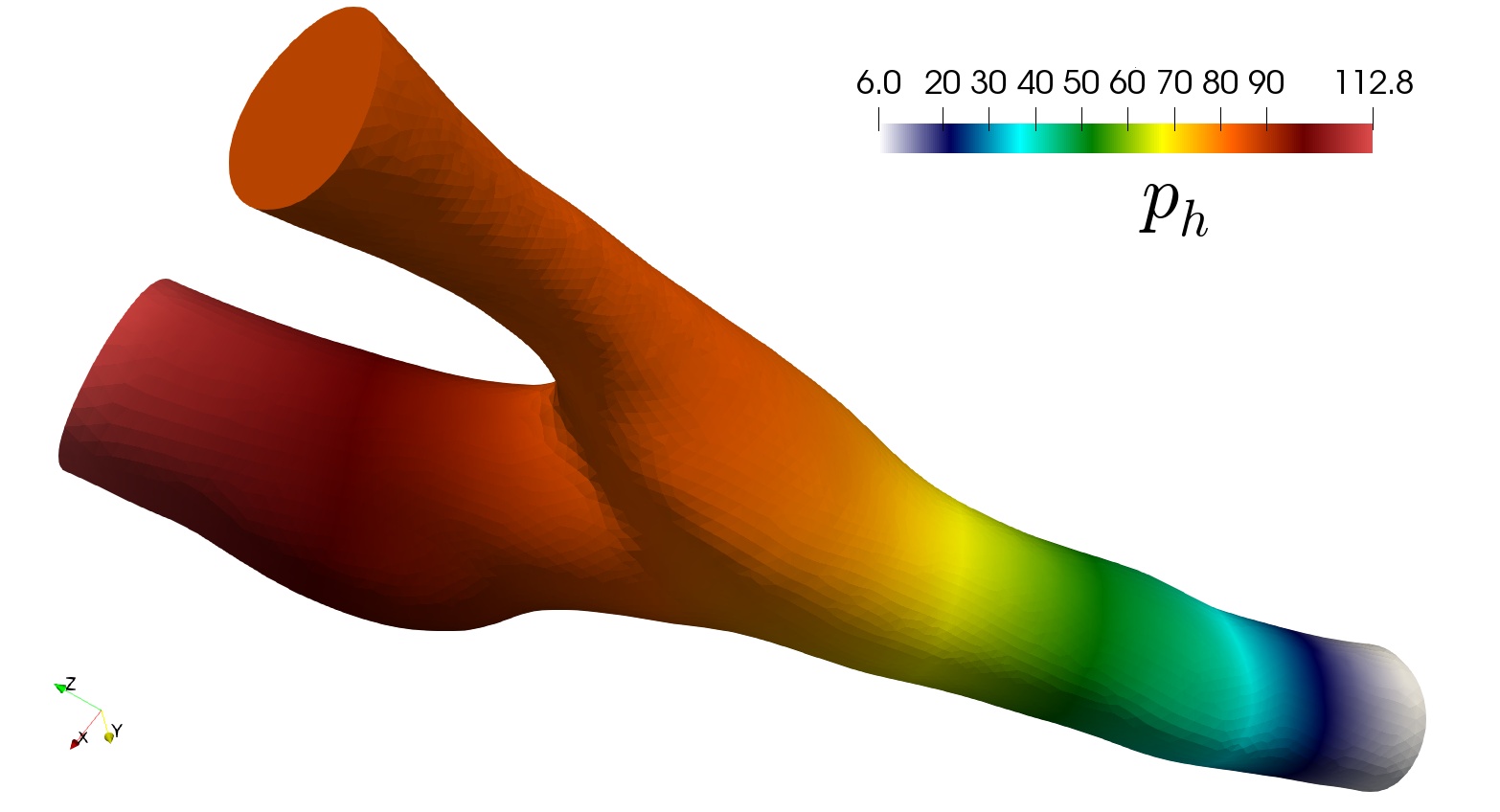}
\includegraphics[width=0.495\textwidth]{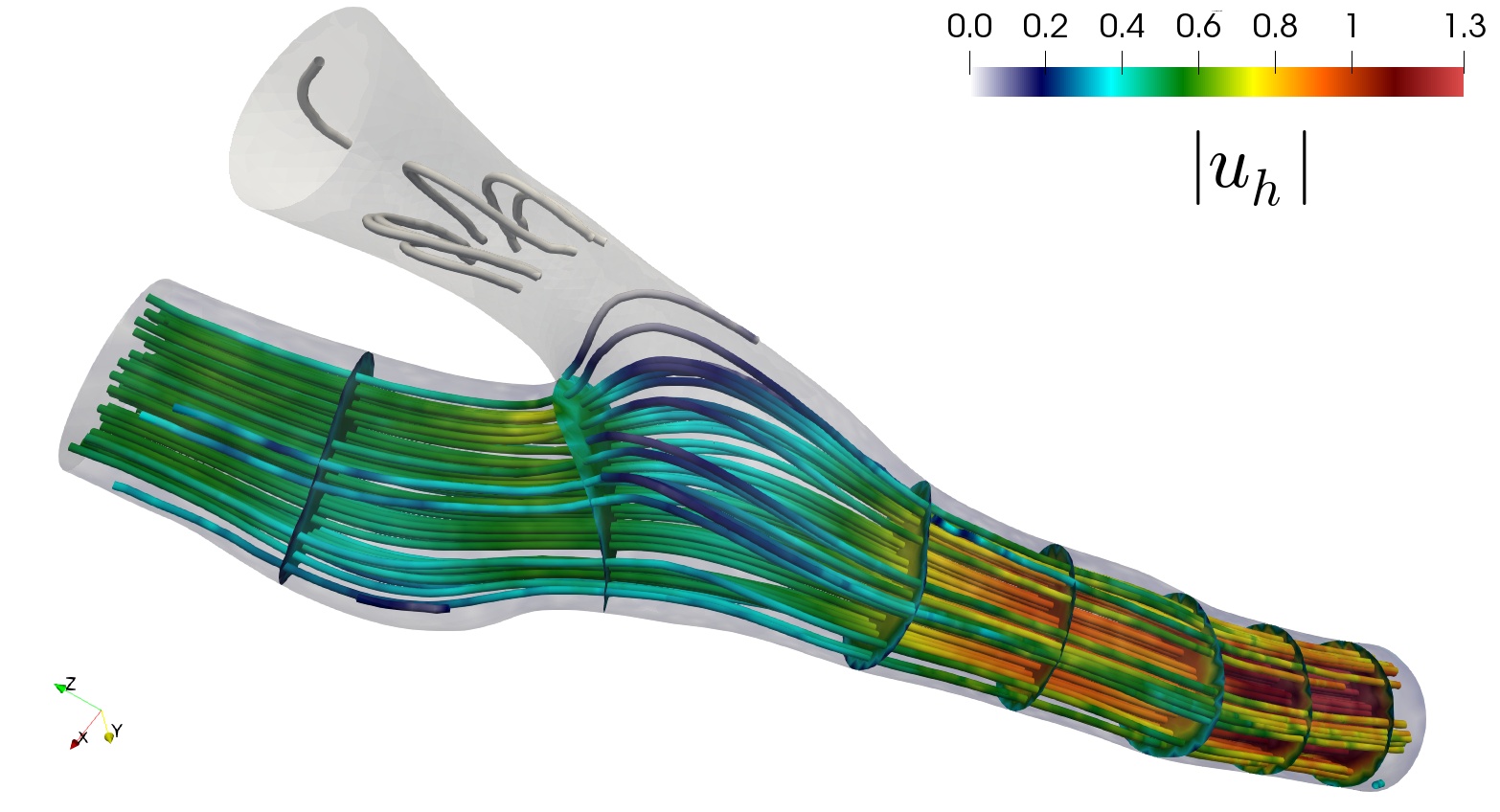}
\end{center}
\caption{Example 5. Bifurcation flow on a femoral bypass geometry. These computations were performed 
using our first-order scheme.}\label{ex04:fig}
\end{figure}

\small\paragraph{Acknowledgments.}
This work has been partially supported by DIUBB through projects 2020127 IF/R and 194608 GI/C, by CONICYT-Chile through the project AFB170001 of the PIA Program: Concurso Apoyo a Centros Cient\'ificos y Tecnol\'ogicos de Excelencia con Financiamiento Basal, and by the HPC-Europa3 Transnational Access programme.   

\end{document}